\documentclass[pdflatex,sn-mathphys-num]{sn-jnl}

\usepackage{graphicx}%
\usepackage{multirow}%
\usepackage{amsmath,amssymb,amsfonts}%
\usepackage{amsthm}%
\usepackage{mathrsfs}%
\usepackage[title]{appendix}%
\usepackage{xcolor}%
\usepackage{textcomp}%
\usepackage{manyfoot}%
\usepackage{booktabs}%
\usepackage{algorithm}%
\usepackage{algorithmicx}%
\usepackage{algpseudocode}%
\usepackage{listings}%
\usepackage{tikz}
\usetikzlibrary{arrows.meta,positioning,fit,backgrounds,calc}

\theoremstyle{thmstyleone}
\newtheorem{theorem}{Theorem}
\newtheorem{lemma}[theorem]{Lemma}
\newtheorem{proposition}[theorem]{Proposition}%

\theoremstyle{thmstyletwo}

\newtheorem{remark}{Remark}
\newtheorem{conjecture}[theorem]{Conjecture}

\newtheorem{problem}[theorem]{Problem}
\theoremstyle{thmstylethree}
\newtheorem{definition}{Definition}

\begin{document}

\title[Criticality in indirect chemotaxis]
{Elliptic criticality versus Volterra memory in indirect chemotaxis cascades}

\author*[1]{\fnm{Louis Shuo} \sur{Wang}}
\email{wang.s41@northeastern.edu}

\affil[1]{\orgdiv{Department of Mathematics},
\orgname{Northeastern University},
\orgaddress{
\city{Boston},
\state{MA},
\postcode{02115},
\country{USA}
}}

\abstract{Indirect signal production is often treated as a higher-order variant of classical Keller–Segel chemotaxis, but its critical structure depends strongly on how the signal cascade is closed.  This paper separates two asymptotic regimes of a two-stage signalling mechanism.  In the parabolic–elliptic–elliptic limit, the chemoattractant is generated by the self-adjoint fourth-order operator $K_\tau=(I-\tau\Delta)^{-1}(I-\Delta)^{-1}$.  We prove its spectral positivity, entropy-dissipation structure, fourth-order principal scaling, and logarithmic kernel singularity in four dimensions.  Consequently, the correct critical space is $L^{N/4}$, and $N=4$ is the mass-critical dimension.  A concentration calculation identifies the natural threshold candidate $M_* = 64\pi^2\tau/\chi$, while the sharp threshold theorem is formulated as an Adams/logarithmic-HLS open problem.  In contrast, the mixed elliptic–parabolic cascade cannot be reduced to a static fourth-order kernel.  Its eliminated signal is a Volterra memory operator whose near-diagonal multiplier has the same order as the classical Keller–Segel drift.  Thus its critical theory must be based on mixed space–time estimates, not static elliptic scaling. Numerical experiments support our operator-level distinction.}


\maketitle

\maketitle

\section{Introduction, Main Results, and Mathematical Positioning}
\label{sec:intro-main-results}

\subsection{Background and motivation}
\label{ssec:introduction_background}

\begin{figure}[htbp]
    \centering
    \includegraphics[width=0.9\linewidth]{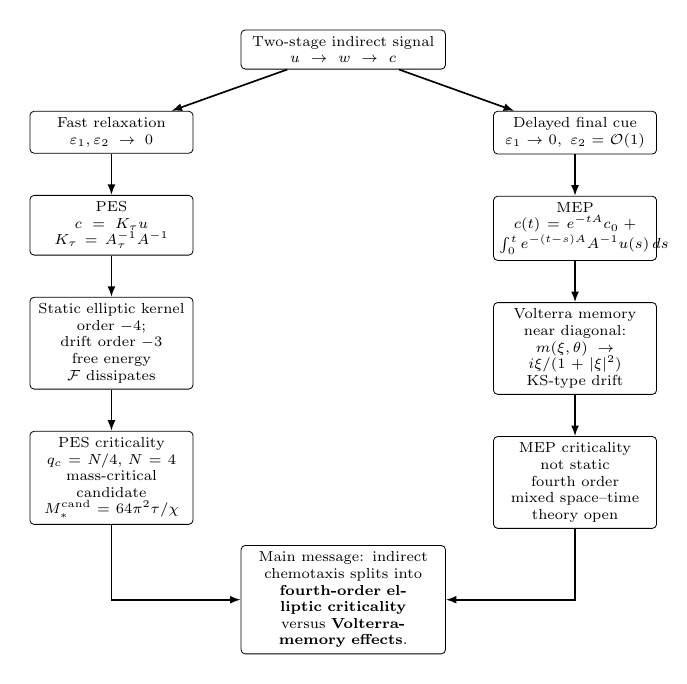}
\caption{Logical framework of the paper.  A two-stage signalling cascade has
two distinct asymptotic closures: PES yields a static fourth-order elliptic
criticality, whereas MEP yields a Volterra-memory mechanism with classical
near-diagonal drift.}
\label{fig:logical-framework}
\end{figure}

The classical Keller--Segel (KS) theory \cite{keller1970initiation,keller1971model,hillen2009user,bellomo2015toward,winkler2016two,liu2025bidirectional,owolabi2026chemotaxis,kiselev2022chemotaxis} is founded on the parabolic--elliptic system:
\begin{equation}
\label{eq:KS-intro}
\partial_t u=\Delta u-\chi\nabla\cdot(u\nabla v),\qquad -\Delta v=u.
\end{equation}
This yields a drift $\nabla v=\nabla(-\Delta)^{-1}u$ that is a Fourier multiplier of order $-1$. Under the parabolic scaling $u_\lambda(x,t)=\lambda^\alpha u(\lambda x,\lambda^2t)$, diffusion ($\lambda^{\alpha+2}$) and drift ($\lambda^{2\alpha}$) balance at $\alpha=2$, giving $\|u_\lambda(\cdot,t)\|_{L^1}=\lambda^{2-N}\|u(\cdot,\lambda^2t)\|_{L^1}$. Thus, $N=2$ is the mass-critical dimension ($q_c=N/2$) for the classical Keller--Segel theory. This 2D criticality \cite{jager1992explosions,wang2025analysis,nagai1995blow,nagai1997application,blanchet2006two,blanchet2008infinite,wang2025analysis1,buseghin2026existence} is governed by the logarithmic singularity of $(-\Delta)^{-1}$ and the sharp $8\pi/\chi$ HLS threshold \cite{carlen1992competing,yu2026pattern,blanchet2006two,blanchet2008infinite,wang2026algebraic,lieb2001analysis,carlen2025stability,ngo2025optimal}, while higher-dimensional dynamics depend heavily on specific structural settings \cite{winkler2010aggregation,wang2026breakdown,winkler2013finite,biler1998local,biler2011large,yu2026rigorous}.

Biological mechanisms like \textit{Dictyostelium} relay \cite{garcia2009group,morimoto2024ion,hayashida2026establishing,hirose2025dictyostelium}, bacterial memory \cite{scanlon2025exploring,othmer2013excitation,zhang2023bacterial,tu2013quantitative,gosztolai2020cellular}, and tumor invasion \cite{roussos2011chemotaxis,liu2023angiogenic,wang2023resveratrol} necessitate moving beyond instantaneous signalling to indirect signal production models \cite{tao2017critical,winkler2022family,chen2024negligibility}. This shift fundamentally alters the interaction kernel's criticality. A static elliptic cascade changes the interaction order, creating a four-dimensional critical regime \cite{fujie2017application,fujie2019blowup,yu2026beyond,hosono2025global} rooted in higher-order potential theory and sharp Adams inequalities \cite{adams1988sharp,adams2003sobolev,hebey2000nonlinear,liang2025global,lin1998classification,martinazzi2009classification,ruf2013sharp}. In contrast, a parabolic signal equation generates a Volterra memory operator requiring a semigroup and maximal-regularity framework \cite{pazy2012semigroups,henry1981geometric,jin2025infinite,lunardi2012analytic,wang2026damage,amann1995linear,pruss2016moving,mao2026global,dai2026critical} to handle its near-diagonal singularity and distributed memory effects \cite{xiang2025critical,yang2026critical,laurenccot2024singular,cai2026optimal,mao2025critical,mao2025finite1,shi2021spatial}.

Ultimately, these indirect mechanisms can induce infinite-time aggregation \cite{tao2017critical}, overcoming classical finite-time blow-up. Rigorous links between indirect and direct signalling have been established via singular limits \cite{bui2026parabolic,yu2026microscopic,liu2025global}. Finally, analyzing Dirac-type aggregation \cite{nagai2000chemotactic,wang2025multi,mao2024dirac} and continuing solutions beyond blow-up relies fundamentally on the weak solution frameworks developed for chemotaxis systems \cite{senba2002weak}.

\subsection{Dimensional derivation and two models}
\label{subsec:dimensional-derivation}

To situate our framework (Figure~\ref{fig:logical-framework}) in the broad literature in Section~\ref{ssec:introduction_background}, we model the cascade by the dimensional system
\begin{equation}
\label{eq:dimensional-cascade}
\begin{cases}
    \partial_t u = D_u\Delta u-\chi_0\nabla\cdot(u\nabla c),\\[2pt]
    \varepsilon_1\,\partial_t w = D_w\Delta w-\lambda_w w+\alpha u,\\[2pt]
    \varepsilon_2\,\partial_t c = D_c\Delta c-\lambda_c c+\beta w,
\end{cases}
\end{equation}
where \(D_u,D_w,D_c>0\) are diffusivities, \(\lambda_w,\lambda_c>0\) are
degradation rates, \(\alpha,\beta>0\) are production rates, and
\(\varepsilon_1,\varepsilon_2>0\) are the relaxation times of the two signal
stages. 

$u$ represents the population density, which produces, releases, or
activates an intermediate mediator \(w\). This mediator then generates the
chemoattractant or motility cue \(c\) sensed by the cells. Thus the migration
velocity is determined by \(\nabla c\), but \(c\) is not produced directly by
\(u\).  In biological terms, the cascade
\[
        u \longrightarrow w \longrightarrow c
\]
may represent relay-mediated chemoattractant production, enzymatic activation
of an initially inactive precursor, matrix-bound storage followed by release of
a soluble cue, or extracellular degradation/conversion processes that separate
the source population from the final sensed signal.
Each mediator stage has an intrinsic length scale
\(\ell_w=\sqrt{D_w/\lambda_w}\) and \(\ell_c=\sqrt{D_c/\lambda_c}\), the
diffusion length over which it equilibrates against its own degradation.
Nondimensionalizing space by \(\ell_w\) turns the mediator stage into
\(I-\Delta\) and the final-cue stage into \(I-\tau\Delta\), where
\begin{equation}
\label{eq:tau-definition}
        \tau:=\frac{\ell_c^{2}}{\ell_w^{2}}
              =\frac{D_c/\lambda_c}{D_w/\lambda_w}
\end{equation}
is the squared ratio of the two mediator ranges. Thus \(\tau\) is not a free
parameter but the relative spatial reach of the final cue against the
intermediate mediator; the production amplitudes \(\alpha/\lambda_w\) and
\(\beta/\lambda_c\) are absorbed into the rescaled sensitivity \(\chi\).

\begin{proposition}[Quasi-steady and memory limits of indirect signalling]
Let \(T_{\rm mig}\) be the characteristic cell-migration time scale.  If both intermediate signal stages relax rapidly compared with cell migration, i.e.,
\[
    \varepsilon_1/T_{\rm mig}\to 0 \quad \text{and} \quad \varepsilon_2/T_{\rm mig}\to 0,
\]
then the two signal equations converge formally to the parabolic--elliptic--elliptic (PES) system:
\begin{equation}
\label{eq:PES-intro}
\begin{cases}
    \partial_t u=\Delta u-\chi\nabla\cdot(u\nabla c),\\
    0=\Delta w-w+u,\\
    0=\tau\Delta c-c+w.
\end{cases}
\tag{PES}
\end{equation}
If instead the intermediate mediator \(w\) is quasi-steady (\(\varepsilon_1/T_{\rm mig}\to 0\)) while the final sensed cue \(c\) relaxes on a time scale comparable to cell migration (\(\varepsilon_2/T_{\rm mig}=\mathcal O(1)\)), then the limiting model is the mixed elliptic--parabolic (MEP) system:
\begin{equation}
\label{eq:MEP-intro}
\begin{cases}
    \partial_t u=\Delta u-\chi\nabla\cdot(u\nabla c),\\
    0=\Delta w-w+u,\\
    \partial_t c=\Delta c-c+w.
\end{cases}
\tag{MEP}
\end{equation}
Thus, PES and MEP correspond to different signal-relaxation limits of the same indirect production mechanism.
\end{proposition}

We fix notation once, for use throughout the paper. On each admissible domain
\(\Omega\) write
\begin{equation}
\label{eq:operator-conventions}
        A:=I-\Delta,\qquad A_\tau:=I-\tau\Delta,
\end{equation}
with the self-adjoint Neumann, periodic, or whole-space Bessel realization of
\(-\Delta\). Both \(A\) and \(A_\tau\) are positive, self-adjoint, and
boundedly invertible, and \(-A\) generates the analytic contraction semigroup
\(\{e^{-tA}\}_{t\ge0}\). 

\begin{remark}[Notational Caveat]\label{rmk:notation}
    To avoid ambiguity and ensure consistency, we fix the following conventions for inverse spatial operators throughout the manuscript: we use $(-\Delta)^{-1}$ exclusively to denote the pure inverse Laplacian, particularly in the context of scaling discussions in Section~\ref{sec:operator-scaling-critical-spaces}. In contrast, $A^{-1}$ strictly represents the resolvent of the shifted Laplacian operator. We also retain the Laplacian operator $\Delta$ in Section~\ref{ssec:small_data} to formulate the small-data theory for \eqref{eq:MEP-intro}. 
\end{remark}

From a physical perspective, the PES and MEP models represent distinct asymptotic limits of the same indirect production mechanism. The PES model corresponds to the quasi-steady limit, describing the precise regime where both extracellular signal-processing stages are fast relative to cell motion. Hence, PES is not meant to describe every indirect chemotaxis system, but rather this specific fast-relaxation cascade. Equivalently, since \(Aw=u\) and \(A_\tau c=w\), we can eliminate the intermediate variables to yield \(c=K_\tau u\), where \(K_\tau:=A_\tau^{-1}A^{-1}\) is a static kernel. 

By contrast, the MEP model corresponds to the delayed-scaling limit. While the intermediate variable is still given by the static relation \(w=A^{-1}u\), the final cue \(c\) is no longer obtained by a static operator. Instead, it satisfies the evolution equation \(\partial_t c+Ac=A^{-1}u\), which integrates to the Volterra memory representation
\begin{equation}
\label{eq:MEP-volterra-intro}
c(t)=e^{-tA}c_0+\int_0^t e^{-(t-s)A}A^{-1}u(s)\,ds.
\end{equation}
Consequently, MEP is governed by a time-memory Volterra operator rather than the static kernel \(K_\tau\), shifting its analysis to a mixed elliptic--parabolic and semigroup framework \cite{pazy2012semigroups,henry1981geometric,lunardi2012analytic,pruss2016moving}. Under this structure, the final cue retains a history of the previous population distribution. Although this memory mechanism may delay aggregation, smooth short-time fluctuations, and alter the onset time or morphology of concentration patterns, it does not by itself justify a static fourth-order critical scaling, because the near-diagonal limit of the Volterra kernel recovers the classical KS first-order drift singularity.

The following distinctions are essential.  First, the PES free-energy identity \eqref{eq:PES-dissipation-intro} is proved for smooth positive solutions and is structurally tied to the self-adjoint operator \(K_\tau=A_\tau^{-1}A^{-1}\).  Second, the PES scaling classification is proved at the level of the principal fourth-order elliptic interaction.  Since \(I-\Delta\) and \(I-\tau\Delta\) contain lower-order terms, this scaling is not an exact global invariance of the full Bessel system; it is the local and high-frequency critical scaling.  Third, the MEP Volterra representation \eqref{eq:MEP-volterra-intro} and multiplier formula \eqref{eq:MEP-multiplier-intro} are derived directly from the parabolic signal equation.  Fourth, the paper proves that MEP has a near-diagonal drift singularity of order \(-1\), and therefore does not assert that MEP belongs to the fourth-order PES critical class.  Finally, a three-dimensional mass-critical theory for MEP is conjectural unless one proves an additional memory mechanism that changes the effective nonlinear criticality.

\subsection{Main Results and Novelty}
\label{ssec:introduction_results}
The paper proves the following structural statements.

\begin{enumerate}[label=\textbf{Theorem \Alph*.}]

\item \textbf{PES has a self-adjoint nonlocal free energy.}
For smooth positive solutions of \eqref{eq:PES-intro}, define
\begin{equation}
\label{eq:PES-energy-intro}
        \mathcal F_{\mathrm{PES}}[u]
        :=
        \int_\Omega u\log u\,dx
        -
        \frac{\chi}{2}\int_\Omega uK_\tau u\,dx .
\end{equation}
Since \(K_\tau=A_\tau^{-1}A^{-1}\) is self-adjoint and positive, the first variation is \(\delta\mathcal F_{\mathrm{PES}}/\delta u=\log u-\chi K_\tau u\), up to an irrelevant additive constant.  The PES equation can be written as \(\partial_tu=\nabla\cdot(u\nabla(\log u-\chi K_\tau u))\).  Consequently,
\begin{equation}
\label{eq:PES-dissipation-intro}
        \frac{d}{dt}\mathcal F_{\mathrm{PES}}[u(t)]
        =
        -
        \int_\Omega u\left|\nabla(\log u-\chi K_\tau u)\right|^2dx
        \le0 .
\end{equation}
Thus PES is a nonlocal entropy-dissipation system with a fourth-order self-adjoint interaction kernel.  This is the fourth-order analogue of the classical KS free-energy mechanism \cite{nagai1995blow,blanchet2006two,blanchet2008infinite}.

\item \textbf{PES has fourth-order critical scaling (Remark~\ref{rmk:notation}).}
In the principal homogeneous scaling regime, \(c\sim(-\Delta)^{-2}u\).  Under \(u_\lambda(x,t)=\lambda^\alpha u(\lambda x,\lambda^2t)\), one has \(c_\lambda\sim\lambda^{\alpha-4}c(\lambda x,\lambda^2t)\) and \(\nabla c_\lambda\sim\lambda^{\alpha-3}\).  Hence \(\nabla\cdot(u_\lambda\nabla c_\lambda)\sim\lambda^{2\alpha-2}\), while \(\Delta u_\lambda\sim\lambda^{\alpha+2}\).  Drift-diffusion balance gives \(\alpha+2=2\alpha-2\), hence \(\alpha=4\).  Therefore \(\|u_\lambda\|_{L^1}=\lambda^{4-N}\|u\|_{L^1}\), so the PES mass-critical dimension is \(N=4\).  More generally, \(\|u_\lambda\|_{L^q}=\lambda^{4-N/q}\|u\|_{L^q}\), and the scaling-critical Lebesgue exponent is \(q_c=N/4\).  Thus the PES critical space is \(L^{N/4}\).

\item \textbf{MEP has a classical near-diagonal drift singularity.}
For MEP, the memory contribution to the drift is
\[
        \nabla c(t)=\nabla e^{-tA}c_0+\int_0^t\nabla e^{-(t-s)A}A^{-1}u(s)\,ds .
\]
In the whole-space Fourier representation, the multiplier associated with the memory kernel at time lag \(\theta=t-s>0\) is
\begin{equation}
\label{eq:MEP-multiplier-intro}
        m(\xi,\theta)=\frac{i\xi e^{-\theta(1+|\xi|^2)}}{1+|\xi|^2}.
\end{equation}
For each fixed \(\theta>0\), the factor \(e^{-\theta(1+|\xi|^2)}\) gives parabolic smoothing.  However, as \(\theta\downarrow0\), \(m(\xi,\theta)\to i\xi/(1+|\xi|^2)\).  The near-diagonal spatial order is therefore \(-1\), the same order as the classical KS drift \(\nabla(I-\Delta)^{-1}u\).  Hence MEP cannot be classified as a static fourth-order elliptic system by analogy with PES.

\item \textbf{Memory-modified criticality for MEP remains open.}
This paper does not prove that MEP has a three-dimensional mass-critical threshold.  Any such result would require a mechanism showing that the Volterra memory term improves the effective critical singularity beyond the near-diagonal KS order.  Establishing or disproving such a mechanism is left as an open problem.
\end{enumerate}

We compare PES and MEP with classical KS models in Table~\ref{tab:KS-PES-MEP-comparison}.
The novelty of the framework is a structural separation of two different mechanisms.

For PES, the two elliptic signal equations generate the self-adjoint operator \(K_\tau=(I-\tau\Delta)^{-1}(I-\Delta)^{-1}\), which is fourth order.  In dimension \(N=4\), this places the model in the Adams/logarithmic-kernel class.  The associated entropy-dissipation structure and the scaling-critical space \(L^{N/4}\) are consistent with that classification.  The relevant analytical background is the endpoint theory for higher-order Sobolev embeddings and exponential integrability \cite{adams1988sharp,adams2003sobolev,lin1998classification,martinazzi2009classification}.

For MEP, the parabolic signal equation prevents reduction to a static fourth-order elliptic kernel.  The eliminated form is the Volterra formula \(c(t)=e^{-tA}c_0+\int_0^t e^{-(t-s)A}(I-\Delta)^{-1}u(s)\,ds\). Although the heat factor smooths the signal for every fixed positive time lag, the near-diagonal limit \(s\uparrow t\) recovers the first-order KS drift singularity.  Therefore MEP belongs to a distinct memory class, and its critical behavior cannot be inferred from the static fourth-order scaling of PES.

In this sense, the paper provides a rigorous structural correction of higher-order indirect chemotaxis criticality: PES belongs to a four-dimensional Adams/logarithmic-kernel class, while MEP belongs to a Volterra-memory class whose critical behavior remains open unless a new memory-driven regularization mechanism is proved.

\section{Operator Structure, Scaling, and Critical Spaces}
\label{sec:operator-scaling-critical-spaces}

This section records the operator-theoretic and scaling facts behind the
criticality classification of the PES model. 
Following the convention stated in Remark~\ref{rmk:notation}, we will primarily use the pure inverse Laplacian $(-\Delta)^{-1}$ throughout this section's scaling discussions, while denoting the resolvents of the shifted operators by $A^{-1}$ and $A_\tau^{-1}$.
The essential point is that the PES signal law is genuinely fourth order: after eliminating the intermediate chemical, one obtains \(c=K_\tau u\), where
\[
        K_\tau = A_\tau^{-1} A^{-1}.
\]
At high frequencies \(K_\tau\) behaves like \(\tau^{-1}(-\Delta)^{-2}\).
Consequently \(\nabla K_\tau\) has order \(-3\), and the nonlinear drift
\(\nabla\!\cdot(u\nabla K_\tau u)\) balances diffusion only under the
amplitude scaling \(u_\lambda(x,t)=\lambda^4u(\lambda x,\lambda^2t)\).
Thus the PES mass-critical dimension is \(N=4\), and the associated
scaling-critical Lebesgue exponent is \(q_c=N/4\).

We first state the standing assumptions.
Throughout the paper, \(N\ge1\), \(\chi>0\), and \(\tau>0\).  The spatial domain \(\Omega\) is one of the following: \(\Omega=\mathbb R^N\), with sufficient decay at infinity; \(\Omega=\mathbb T^N\), with periodic boundary conditions; or a bounded connected \(C^2\) domain \(\Omega\subset\mathbb R^N\), equipped with no-flux boundary conditions.  The elliptic and parabolic operators are interpreted through the corresponding self-adjoint or sectorial realizations, as standard in elliptic and analytic-semigroup theory \cite{gilbarg1998elliptic,evans2022partial,pazy2012semigroups,gao2022rolling,lunardi2012analytic}.

In the bounded-domain case, the no-flux conditions are \((\nabla u-\chi u\nabla c)\cdot\nu=0\), \(\partial_\nu w=0\), and \(\partial_\nu c=0\) on \(\partial\Omega\), where \(\nu\) is the outward unit normal.  For the classical KS Poisson equation on a bounded domain, we use the normalized Neumann formulation \(-\Delta v=u-\overline u\), \(\partial_\nu v=0\), and \(\int_\Omega v\,dx=0\), where \(\overline u=|\Omega|^{-1}\int_\Omega u\,dx\).

Initial data satisfy \(u_0\ge0\), \(u_0\not\equiv0\), and \(u_0\in L^1(\Omega)\cap L^\infty(\Omega)\).  For MEP, we additionally prescribe \(c(0)=c_0\), \(c_0\ge0\), with \(c_0\) regular enough to justify the computations below; for instance, \(c_0\in W^{1,\infty}(\Omega)\) in the classical setting.  All entropy identities are stated for smooth positive solutions.  Weak or approximate solutions are understood by standard regularization, compactness, and lower-semicontinuity arguments whenever such arguments are available \cite{ladyzhenskaia1968linear,amann1995linear,simon1986compact}.

\subsection{Spectral representation of the PES kernel}

Assume first that \(\Omega\subset\mathbb R^N\) is bounded with \(C^2\)
boundary and homogeneous Neumann boundary conditions.  Let
\(\{(\lambda_j,\varphi_j)\}_{j\ge0}\) be the Neumann spectral resolution of
\(-\Delta\), normalized by \(\{\varphi_j\}_{j\ge0}\) being an orthonormal basis
of \(L^2(\Omega)\), with \(0=\lambda_0\le\lambda_1\le\lambda_2\le\cdots\).
Then, for \(f\in L^2(\Omega)\),
\[
        K_\tau f
        =
        \sum_{j\ge0}
        \frac{\langle f,\varphi_j\rangle_{L^2}}
        {(1+\lambda_j)(1+\tau\lambda_j)}
        \varphi_j .
\]
This representation immediately gives positivity, self-adjointness, and
compactness on \(L^2(\Omega)\).  The same formulas hold on \(\mathbb T^N\)
with the Fourier basis.  On \(\mathbb R^N\), \(K_\tau\) is the Fourier
multiplier with symbol \(((1+|\xi|^2)(1+\tau|\xi|^2))^{-1}\).

\begin{lemma}[Spectral positivity, self-adjointness, and smoothing]
\label{lem:spectral-positivity}
Let \(\tau>0\).  On a bounded smooth Neumann domain, \(K_\tau\) is positive,
self-adjoint, compact on \(L^2(\Omega)\), and smoothing of order four.  More
precisely, for every \(s\in\mathbb R\) there is \(C=C(s,\tau,\Omega)>0\) such
that \(\|K_\tau f\|_{H^{s+4}}\le C\|f\|_{H^s}\).
\end{lemma}

\begin{proof}
By the spectral formula, \(K_\tau\varphi_j=m_j\varphi_j\), where
\(m_j=((1+\lambda_j)(1+\tau\lambda_j))^{-1}>0\).  Hence
\(\langle K_\tau f,g\rangle=\langle f,K_\tau g\rangle\) and
\(\langle K_\tau f,f\rangle=\sum_{j\ge0}m_j|\langle f,\varphi_j\rangle|^2\ge0\).
Since \(m_j\to0\), \(K_\tau\) is compact on \(L^2(\Omega)\).  Finally,
\((1+\lambda_j)^2m_j\) is uniformly bounded above by a constant depending only
on \(\tau\), which gives the asserted \(H^s\to H^{s+4}\) estimate.
\end{proof}

In the whole-space representation,
\[
        \widehat{K_\tau f}(\xi)
        =
        \frac{\widehat f(\xi)}
        {(1+|\xi|^2)(1+\tau|\xi|^2)}.
\]
For \(|\xi|\gg1\), the multiplier satisfies
\[
        \frac{1}{(1+|\xi|^2)(1+\tau|\xi|^2)}
        =
        \frac{1}{\tau|\xi|^4}\bigl(1+O(|\xi|^{-2})\bigr).
\]
\begin{lemma}[High-frequency homogeneity]
\label{lem:high-frequency-homogeneity}
The PES interaction operator \(K_\tau\) is a fourth-order elliptic smoothing
operator.  Its high-frequency principal part is \(\tau^{-1}(-\Delta)^{-2}\).
Consequently \(\nabla K_\tau\) is of order \(-3\).
\end{lemma}

\begin{proof}
On \(\mathbb R^N\), the assertion follows directly from the multiplier
asymptotic above.  On a compact manifold with boundary and Neumann realization,
the same conclusion follows from the spectral formula:
\[
        \frac{1}{(1+\lambda_j)(1+\tau\lambda_j)}
        =
        \frac{1}{\tau\lambda_j^2}\bigl(1+O(\lambda_j^{-1})\bigr)
        \quad\text{as }j\to\infty .
\]
Thus \(K_\tau\) has principal order \(-4\).  Taking one spatial derivative gives
order \(-3\).
\end{proof}

Thus, \(K_\tau\) is a fourth-order elliptic Bessel-type operator with a principal high-frequency order of \(-4\), and the drift operator \(\nabla K_\tau\) has order \(-3\). This high-frequency structure places PES in the same broad analytic family as fourth-order Bessel-potential and Adams-type problems \cite{adams1988sharp,adams2003sobolev,stein1970singular,triebel1995interpolation}, serving as the fundamental source of its four-dimensional criticality.

\subsection{Critical scaling as an exact symmetry of the principal flow}
\label{subsec:exact-scaling}
The homogeneity balance below is not merely dimensional: it is the exact
scaling symmetry of the principal part of the PES flow. Recall from
Lemma~\ref{lem:high-frequency-homogeneity} that
\(K_\tau=\tau^{-1}(-\Delta)^{-2}+L_\tau\), where \(L_\tau\) collects strictly
lower-order terms (symbol \(O(|\xi|^{-6})\) at high frequency). Discarding
\(L_\tau\) gives the principal PES flow
\begin{equation}
\label{eq:principal-PES}
        \partial_t u=\Delta u-\frac{\chi}{\tau}
        \nabla\!\cdot\!\bigl(u\,\nabla(-\Delta)^{-2}u\bigr).
\end{equation}

\begin{lemma}[Exact scaling invariance of the principal flow]
\label{lem:exact-scaling}
For \(\lambda>0\) let \(u_\lambda(x,t)=\lambda^4u(\lambda x,\lambda^2t)\). Then
\(u\) solves \eqref{eq:principal-PES} if and only if \(u_\lambda\) does. The
amplitude \(\alpha=4\) is the unique exponent for which this holds.
\end{lemma}

\begin{proof}
Since \((-\Delta)^{-2}\) has symbol \(|\xi|^{-4}\), it is homogeneous of degree
\(-4\): for \(g=f(\lambda\,\cdot)\) one has
\((-\Delta)^{-2}g=\lambda^{-4}\bigl((-\Delta)^{-2}f\bigr)(\lambda\,\cdot)\).
Writing \(y=\lambda x\), \(s=\lambda^2t\) and \(u_\lambda=\lambda^\alpha
u(\lambda x,\lambda^2t)\), each term of \eqref{eq:principal-PES} scales as a
pure power of \(\lambda\):
\[
   \partial_t u_\lambda\sim\lambda^{\alpha+2},\qquad
   \Delta u_\lambda\sim\lambda^{\alpha+2},\qquad
   \nabla\!\cdot\!\bigl(u_\lambda\nabla(-\Delta)^{-2}u_\lambda\bigr)
   \sim\lambda^{2\alpha-2}.
\]
The three powers coincide iff \(\alpha+2=2\alpha-2\), i.e.\ \(\alpha=4\); for
that value every term carries the common factor \(\lambda^{6}\) and
\eqref{eq:principal-PES} is preserved.
\end{proof}

\begin{remark}[The lower-order terms are subcritical]
\label{rem:perturbation-subcritical}
The full kernel is not scale invariant, but the defect is subleading exactly in
the regime that decides criticality. Under \(u\mapsto u_\lambda\) the drift
samples \(K_\tau\) at frequencies \(\sim\lambda\); there
\[
   \widehat{K_\tau}(\lambda\xi)
   =\frac{1}{(1+\lambda^2|\xi|^2)(1+\tau\lambda^2|\xi|^2)}
   =\frac{1}{\tau\lambda^4|\xi|^4}\bigl(1+O(\lambda^{-2})\bigr),
\]
so the correction relative to the principal balance is \(O(\lambda^{-2})\) and
vanishes as \(\lambda\to\infty\), i.e.\ in the concentration limit. Hence
\eqref{eq:principal-PES} is the exact scaling model and PES is an
\(O(\lambda^{-2})\) perturbation of it at high frequencies. This is the precise
sense in which \(\alpha=4\) is critical.
\end{remark}

Under the PES critical scaling \(u_\lambda(x,t)=\lambda^4u(\lambda
x,\lambda^2t)\), the \(L^q\)-norm transforms as
\[
        \|u_\lambda(\cdot,t)\|_{L^q}
        =
        \lambda^{4-\frac{N}{q}}
        \|u(\cdot,\lambda^2t)\|_{L^q}.
\]
Thus the scale-invariant exponent is determined by \(4-N/q=0\), namely
\[
        q_c=\frac{N}{4}.
\]
In particular, \(N=4\) is the mass-critical dimension, because then \(q_c=1\)
and \(\|u_\lambda\|_{L^1}=\|u\|_{L^1}\).  For \(N<4\), the mass is subcritical;
for \(N=4\), \(L^1\) is critical; and for \(N>4\), the natural Lebesgue
critical space is \(L^{N/4}\).

\begin{proposition}[Scaling-critical exponent]
\label{prop:critical-exponent}
The PES scaling-critical Lebesgue exponent is \(q_c=N/4\).  Consequently, the
mass-critical dimension is \(N=4\).
\end{proposition}

\begin{proof}
The preceding scaling identity gives
\(\|u_\lambda\|_{L^q}=\lambda^{4-N/q}\|u\|_{L^q}\).  The exponent is invariant
if and only if \(4-N/q=0\), equivalently \(q=N/4\).  Since \(L^1\) corresponds
to \(q=1\), mass is invariant exactly when \(N=4\).
\end{proof}

\begin{remark}[Correction of the classical exponent]
\label{rem:corrected-critical-space}
For the classical KS model, the critical amplitude is \(\alpha=2\),
and hence \(q_c=N/2\).  For PES, the fourth-order cascade changes the amplitude
to \(\alpha=4\), and hence the critical exponent is \(q_c=N/4\).  Thus
\(L^{N/2}\) is not the PES critical Lebesgue space.
\end{remark}

\subsection{The logarithmic kernel in four dimensions}

The dimension \(N=4\) also appears from the Green-kernel viewpoint.  The
fundamental solution of \((-\Delta)^2\) in \(\mathbb R^N\) has power-law
behavior \(|x|^{4-N}\) when \(N\ne4\), while in the borderline dimension
\(N=4\) it is logarithmic:
\[
        \Phi_4(x)
        =
        \frac{1}{8\pi^2}\log\frac1{|x|}
        \quad
        \text{modulo an additive smooth term}.
\]
Since \(K_\tau\) has principal part \(\tau^{-1}(-\Delta)^{-2}\), its kernel has
the same leading local singularity, with coefficient scaled by \(\tau^{-1}\).
Thus in \(N=4\), the PES interaction is governed by a logarithmic kernel, just
as the classical KS interaction \((-\Delta)^{-1}\) is logarithmic in
\(N=2\).

\begin{lemma}[Four-dimensional logarithmic kernel]
\label{lem:log-kernel-four-dim}
Let \(G_\tau\) denote the local kernel of \(K_\tau\).  In dimension \(N=4\), its
leading near-diagonal singularity is logarithmic:
\[
        G_\tau(x,y)
        =
        \frac{1}{8\pi^2\tau}\log\frac1{|x-y|}
        +
        R_\tau(x,y),
\]
where \(R_\tau\) is less singular than the logarithmic term locally away from
the boundary.  In dimensions \(N\ne4\), the corresponding principal singularity
is of power-law type \(|x-y|^{4-N}\).
\end{lemma}

\begin{proof}
Fix $N=4$ and work locally near the diagonal; boundary and lower-order
contributions are absorbed into $R_\tau$ by standard elliptic parametrix
theory.

Assume first $\tau\neq1$. Partial fractions in the symbol variable give the
operator identity
\[
   K_\tau=(I-\Delta)^{-1}(I-\tau\Delta)^{-1}
        =\frac{1}{1-\tau}(I-\Delta)^{-1}
         -\frac{\tau}{1-\tau}(I-\tau\Delta)^{-1},
\]
which exhibits $K_\tau$ as a difference of two order $-2$ Bessel
potentials. Each summand is individually as singular as $(-\Delta)^{-1}$, so
the asserted order $-4$ behaviour can only arise from a cancellation; we make
it explicit.

Let $G^{(\mu)}$ denote the kernel of $(I-\mu\Delta)^{-1}$, $\mu>0$. Writing
$G^{(\mu)}=\tfrac{1}{4\pi^2\mu}|x|^{-2}+h_\mu$ and using
$-\Delta\bigl(\tfrac{1}{4\pi^2}|x|^{-2}\bigr)=\delta$ in $\mathbb R^4$, one
finds $(I-\mu\Delta)h_\mu=-\tfrac{1}{4\pi^2\mu}|x|^{-2}$. Since
$\Delta\log|x|=2|x|^{-2}$ in $\mathbb R^4$, the leading local solution is
$h_\mu=\tfrac{1}{8\pi^2\mu^2}\log|x|+\mathcal O(1)$. Hence, near the diagonal,
\[
   G^{(1)}(x)=\frac{1}{4\pi^2}|x|^{-2}+\frac{1}{8\pi^2}\log|x|+\mathcal O(1),
   \qquad
   G^{(\tau)}(x)=\frac{1}{4\pi^2\tau}|x|^{-2}
                +\frac{1}{8\pi^2\tau^2}\log|x|+\mathcal O(1).
\]
Substituting into the partial-fraction identity, the $|x|^{-2}$ coefficients
sum to
\[
   \frac{1}{1-\tau}\cdot\frac{1}{4\pi^2}
   -\frac{\tau}{1-\tau}\cdot\frac{1}{4\pi^2\tau}=0,
\]
so the second-order singularities cancel exactly. The surviving logarithmic
coefficient is
\[
   \frac{1}{1-\tau}\cdot\frac{1}{8\pi^2}
   -\frac{\tau}{1-\tau}\cdot\frac{1}{8\pi^2\tau^2}
   =\frac{1}{8\pi^2(1-\tau)}\Bigl(1-\frac1\tau\Bigr)
   =-\frac{1}{8\pi^2\tau}.
\]
Therefore $G_\tau(x,y)=\tfrac{1}{8\pi^2\tau}\log\frac{1}{|x-y|}+R_\tau(x,y)$
with $R_\tau$ bounded near the diagonal, which is the asserted expansion.

The borderline case $\tau=1$ (where the partial fraction degenerates into a
double pole) follows either by continuity in $\tau$ of the coefficient
$\tfrac{1}{8\pi^2\tau}$, or directly: $K_1=(I-\Delta)^{-2}$ has high-frequency
symbol $|\xi|^{-4}$, whose fundamental solution is the biharmonic kernel
$\tfrac{1}{8\pi^2}\log\frac1{|x|}$ in $\mathbb R^4$. In dimensions $N\neq4$ the
same partial-fraction/symbol analysis yields the power-law principal
singularity $|x-y|^{4-N}$.
\end{proof}

\begin{remark}[The analogy]
\label{rem:correct-analogy}
The proper analogy is therefore
\[
        \text{classical KS in }N=2
        \quad\longleftrightarrow\quad
        \text{PES in }N=4.
\]
The first case is governed by the logarithmic kernel of \((-\Delta)^{-1}\) in
two dimensions, while the second is governed by the logarithmic kernel of
\((-\Delta)^{-2}\) in four dimensions.
\end{remark}

The preceding lemmas establish the operator-level basis for the revised
criticality classification.  The PES system is not merely an indirect
KS model with a cosmetic extra variable; its eliminated signal law
is a fourth-order elliptic interaction.  The relevant scaling amplitude is
\(\alpha=4\), the mass-critical dimension is \(N=4\), and the critical
Lebesgue scale is \(L^{N/4}\).  Any sharp threshold theory for PES should
therefore be sought in the four-dimensional logarithmic/Adams regime, rather
than in a five-dimensional mass-critical regime or in the classical
\(L^{N/2}\) KS scale.

This conclusion applies to PES only.  It should not be transferred to the MEP
system without further analysis, because MEP is governed by a Volterra memory
operator rather than by the static fourth-order kernel \(K_\tau\).

\section{Free Energy, Adams-Type Control, and the Four-Dimensional Critical Regime}
\label{sec:free-energy-adams}

We now turn to the analytic mechanism behind the four-dimensional PES
critical theory.  Throughout this section \(K_\tau=(I-\tau\Delta)^{-1}
(I-\Delta)^{-1}\), with \(\tau>0\), and the PES equation is written in the
eliminated form
\begin{equation}
        \partial_t u=\Delta u-\chi\nabla\cdot(u\nabla K_\tau u).
\label{PES-eliminated}
\end{equation}
The competition is between the entropy \(\int_\Omega u\log u\,dx\) and the
fourth-order attractive interaction \(\int_\Omega uK_\tau u\,dx\).  In
dimension \(N=4\), the kernel of \(K_\tau\) has a logarithmic leading
singularity, and this places the problem in the Adams/logarithmic
HLS regime.

For \(M>0\), define the admissible entropy class
\begin{equation}
        \mathcal A_M(\Omega)
        :=
        \left\{
        u\ge0:\ \int_\Omega u\,dx=M,\quad
        \int_\Omega u\log u\,dx<\infty
        \right\},
\label{admissible-class}
\end{equation}
with the convention \(0\log0=0\).  For \(u\in\mathcal A_M(\Omega)\), the PES
free energy is
\begin{equation}
        \mathcal F_{\mathrm{PES}}[u]
        :=
        \int_\Omega u\log u\,dx
        -
        \frac{\chi}{2}\int_\Omega uK_\tau u\,dx.
\label{PES-free-energy-section3}
\end{equation}
The factor \(1/2\) is essential: since \(K_\tau\) is self-adjoint, the first
variation of the interaction term is \(\chi K_\tau u\), not
\(\chi K_\tau u/2\).  Thus, modulo an irrelevant additive constant,
\begin{equation}
        \frac{\delta\mathcal F_{\mathrm{PES}}}{\delta u}
        =
        \log u-\chi K_\tau u.
\label{PES-first-variation}
\end{equation}
Consequently \eqref{PES-eliminated} becomes the formal Wasserstein-type
gradient-flow identity
\begin{equation}
        \partial_tu
        =
        \nabla\cdot\left(
        u\nabla\frac{\delta\mathcal F_{\mathrm{PES}}}{\delta u}
        \right)
        =
        \nabla\cdot\left(u\nabla(\log u-\chi K_\tau u)\right).
\label{PES-gradient-flow}
\end{equation}

\begin{proposition}[Energy dissipation identity]
\label{prop:PES-energy-dissipation}
Let \(u\) be a smooth positive solution of \eqref{PES-eliminated} on
\([0,T]\), with either periodic boundary conditions, sufficient decay on
\(\mathbb R^N\), or no-flux boundary condition
\((\nabla u-\chi u\nabla K_\tau u)\cdot\nu=0\) on \(\partial\Omega\).  Then
\begin{equation}
        \frac{d}{dt}\mathcal F_{\mathrm{PES}}[u(t)]
        =
        -
        \int_\Omega
        u\left|\nabla(\log u-\chi K_\tau u)\right|^2dx
        \le0.
\label{PES-dissipation-section3}
\end{equation}
\end{proposition}

\begin{proof}
Using \eqref{PES-first-variation} and \eqref{PES-gradient-flow}, integration
by parts gives
\[
        \frac{d}{dt}\mathcal F_{\mathrm{PES}}[u(t)]
        =
        \int_\Omega
        (\log u-\chi K_\tau u)\,
        \nabla\cdot\left(u\nabla(\log u-\chi K_\tau u)\right)dx.
\]
The boundary term vanishes by the stated boundary condition, periodicity, or
decay at infinity.  This yields \eqref{PES-dissipation-section3}.
\end{proof}

\begin{remark}[Mathematical physics perspective and Wasserstein framework]
\label{rem:gradient-flow-physics}
The structural identity \eqref{PES-gradient-flow} and the resulting dissipation inequality \eqref{PES-dissipation-section3} firmly embed the PES model within the modern framework of mathematical physics and optimal transport. Formally, the system can be characterized as a metric gradient flow on the Wasserstein space $(\mathcal{P}_2(\Omega), d_W)$ governed by the geometric equation
\[
        \partial_t u = -\operatorname{grad}_{W} \mathcal F_{\mathrm{PES}}[u],
\]
where the total free energy functional $\mathcal F_{\mathrm{PES}}[u] = \mathcal{H}[u] - \mathcal{W}[u]$ represents a delicate balance between the internal Boltzmann entropy $\mathcal{H}[u] := \int_\Omega u \log u \,dx$ and the nonlocal interaction potential energy $\mathcal{W}[u] := \frac{\chi}{2}\int_\Omega u K_\tau u \,dx$. Under this Riemannian-type framework on measures, the energy dissipation rate \eqref{PES-dissipation-section3} corresponds precisely to the squared Wasserstein metric slope (or the squared norm of the metric gradient): with $ V=-\nabla(\log u-\chi K_\tau u)$ the following identity holds
\[
        \left| \operatorname{grad}_{W} \mathcal F_{\mathrm{PES}}[u] \right|^2_{W} 
        = \int_\Omega u \left|\nabla \frac{\delta \mathcal F_{\mathrm{PES}}}{\delta u}\right|^2 dx 
        = \int_\Omega u |V|^2 \,dx.
\]
This rigorous formulation encapsulates a profound thermodynamic competition in four dimensions: the maximization of the entropy $\mathcal{H}[u]$ drives the linear diffusion to disperse the cell population, whereas the minimization of the potential energy $\mathcal{W}[u]$ via the fourth-order Bessel-type kernel $K_\tau$ drives non-local aggregation. 
\end{remark}

\subsection{Entropy versus fourth-order attraction}

The attractive part of \(\mathcal F_{\mathrm{PES}}\) is
\begin{equation}
        \frac{\chi}{2}\int_\Omega uK_\tau u\,dx.
\label{attraction-term}
\end{equation}
In dimension \(N=4\), the local kernel \(G_\tau(x,y)\) of \(K_\tau\) satisfies
\begin{equation}
        G_\tau(x,y)
        =
        \frac{1}{8\pi^2\tau}\log\frac1{|x-y|}
        +
        R_\tau(x,y),
\label{kernel-log-four}
\end{equation}
where \(R_\tau\) is less singular near the diagonal.  Hence the leading
singular part of \(\int uK_\tau u\) is logarithmic.  This is the fourth-order
analogue of the classical two-dimensional KS free energy
\begin{equation}
        \mathcal F_{\mathrm{KS}}[u]
        =
        \int u\log u\,dx
        -
        \frac{\chi}{2}\int u(-\Delta)^{-1}u\,dx,
\label{KS-free-energy-comparison}
\end{equation}
where \((-\Delta)^{-1}\) has a logarithmic kernel in \(N=2\).  The
analogy is therefore
\begin{equation}
        (-\Delta)^{-1}\text{ logarithmic in }N=2,
        \qquad
        (-\Delta)^{-2}\text{ logarithmic in }N=4.
\label{log-analogy}
\end{equation}

\begin{lemma}[Concentration scaling and candidate sharp mass]
\label{lem:candidate-mass}
Let \(N=4\), let \(x_0\in\Omega\), and let \(U\ge0\) be smooth, compactly
supported in \(B_1(0)\), with \(\int_{\mathbb R^4}U\,dx=M\).  Define
\(u_\varepsilon(x)=\varepsilon^{-4}U((x-x_0)/\varepsilon)\).  Then, as
\(\varepsilon\downarrow0\),
\begin{equation}
        \int_\Omega u_\varepsilon\log u_\varepsilon\,dx
        =
        4M\log\frac1\varepsilon+O(1),
\label{entropy-scaling}
\end{equation}
and
\begin{equation}
        \int_\Omega u_\varepsilon K_\tau u_\varepsilon\,dx
        =
        \frac{M^2}{8\pi^2\tau}\log\frac1\varepsilon+O(1).
\label{interaction-scaling}
\end{equation}
Consequently
\begin{equation}
        \mathcal F_{\mathrm{PES}}[u_\varepsilon]
        =
        \left(
        4M-\frac{\chi M^2}{16\pi^2\tau}
        \right)
        \log\frac1\varepsilon+O(1).
\label{energy-scaling}
\end{equation}
The formal concentration threshold suggested by this computation is
\begin{equation}
        M_*=\frac{64\pi^2\tau}{\chi}.
\label{candidate-critical-mass}
\end{equation}
\end{lemma}

\begin{proof}
Set $u_\varepsilon(x)=\varepsilon^{-4}U((x-x_0)/\varepsilon)$ and substitute
$x=x_0+\varepsilon z$, $dx=\varepsilon^4\,dz$, so $\int u_\varepsilon=M$ for
every $\varepsilon$. For the entropy,
\[
   \int_\Omega u_\varepsilon\log u_\varepsilon
   =\int U(z)\bigl(4\log\tfrac1\varepsilon+\log U(z)\bigr)dz
   =4M\log\tfrac1\varepsilon+\mathcal O(1),
\]
which is \eqref{entropy-scaling}.

For the interaction, insert $G_\tau=\tfrac{1}{8\pi^2\tau}\log\frac1{|x-y|}
+R_\tau$ from Lemma~\ref{lem:log-kernel-four-dim}. Two facts make the
remainder harmless under concentration. First, $R_\tau$ is bounded near the
diagonal, so its contribution is
$\iint u_\varepsilon(x)u_\varepsilon(y)R_\tau\,dx\,dy=\mathcal O(M^2)=\mathcal O(1)$, uniformly
in $\varepsilon$. This step relies essentially on the cancellation in
Lemma~\ref{lem:log-kernel-four-dim}: had $K_\tau$ retained the $|x-y|^{-2}$
singularity of either constituent Bessel potential, the corresponding double
integral would scale like $\varepsilon^{-2}$ and would not be $\mathcal O(1)$.
Second, scaling the logarithmic part via
$\log\frac1{|x-y|}=\log\frac1\varepsilon+\log\frac1{|z-\zeta|}$ gives
\[
   \frac{1}{8\pi^2\tau}\iint u_\varepsilon(x)u_\varepsilon(y)
        \log\tfrac1{|x-y|}\,dx\,dy
   =\frac{M^2}{8\pi^2\tau}\log\tfrac1\varepsilon
    +\frac{1}{8\pi^2\tau}\iint U(z)U(\zeta)\log\tfrac1{|z-\zeta|}\,dz\,d\zeta,
\]
whose second term is a fixed finite constant. This is
\eqref{interaction-scaling}. Substituting both expansions into
\eqref{PES-free-energy-section3} yields
\[
   \mathcal F_{\mathrm{PES}}[u_\varepsilon]
   =\Bigl(4M-\frac{\chi M^2}{16\pi^2\tau}\Bigr)\log\tfrac1\varepsilon+\mathcal O(1),
\]
which is \eqref{energy-scaling}; the bracket changes sign at
$M_*=64\pi^2\tau/\chi$.
\end{proof}

\begin{remark}[Two consistency checks for the candidate threshold]
\label{rem:consistency-checks}
The candidate $M_*=64\pi^2\tau/\chi$ is unproven, but it passes two
independent consistency tests that calibrate the heuristic against the
classical second-order theory.

\emph{(i) Duality ratio.} In two dimensions the sharp Moser--Trudinger
constant is $4\pi$ and the proven KS critical mass is $8\pi/\chi$, so
$\chi M_{\rm KS}=8\pi=2\cdot4\pi$. In four dimensions the sharp Adams
constant in \eqref{Adams-inequality} is $32\pi^2$, and at the natural filter
normalization $\tau=1$ the candidate satisfies $\chi M_*=64\pi^2=2\cdot
32\pi^2$. Thus the candidate stands in the same $2{:}1$ relation to the sharp
four-dimensional exponential-integrability constant as the proven KS mass does
to its two-dimensional counterpart.

\emph{(ii) Calibration in the proven case.} The concentration-scaling argument
of Lemma~\ref{lem:candidate-mass} returns the correct value where the answer
is known. Applying the identical computation to two-dimensional KS, with
$u_\varepsilon=\varepsilon^{-2}U((x-x_0)/\varepsilon)$ and the logarithmic
kernel $\tfrac{1}{2\pi}\log\frac1{|x|}$ of $(-\Delta)^{-1}$ in $\mathbb R^2$,
gives
\[
   \mathcal F_{\mathrm{KS}}[u_\varepsilon]
   =\Bigl(2M-\frac{\chi M^2}{4\pi}\Bigr)\log\tfrac1\varepsilon+O(1),
\]
whose bracket changes sign at $M=8\pi/\chi$, the sharp KS threshold. The same
heuristic that produces $64\pi^2\tau/\chi$ here is therefore exact in its
two-dimensional analogue.

Neither observation establishes sharpness; that requires the
$K_\tau$-adapted logarithmic HLS/Adams inequality of
Conjecture~\ref{conj:sharp-PES-threshold}. They show only that $M_*$ is the
value consistent with both the variational duality and the calibrated
concentration heuristic.
\end{remark}

\begin{remark}[Status of the constant \(M_*\)]
\label{rem:status-Mstar}
The number \(M_*=64\pi^2\tau/\chi\) is an exact concentration-scaling
candidate.  It is not asserted here to be the sharp global threshold unless a
sharp logarithmic HLS or Adams-type inequality adapted
to \(K_\tau\), \(\Omega\), and the boundary conditions is proved.  The
distinction between the formal sharp constant and a proved threshold is
essential.
\end{remark}

\subsection{Adams-type control}

The relevant functional inequality in four dimensions is the Adams inequality.
In one standard form, if \(\Omega\subset\mathbb R^4\) is bounded and
\(v\in W^{2,2}_0(\Omega)\) satisfies \(\|\Delta v\|_{L^2}\le1\), then
\begin{equation}
        \int_\Omega e^{32\pi^2 v^2}\,dx
        \le
        C|\Omega|.
\label{Adams-inequality}
\end{equation}
This is the fourth-order analogue of the Trudinger--Moser inequality.  In the
present model, however, the entropy variable is \(u\), whereas the smoothed
potential is \(K_\tau u\).  Therefore the useful form is not merely
\eqref{Adams-inequality}, but a mass-dependent logarithmic interaction bound
of the type
\begin{equation}
        \int_\Omega uK_\tau u\,dx
        \le
        a(M,\tau,\Omega)\int_\Omega u\log u\,dx
        +
        C(M,\tau,\Omega),
\label{abstract-log-bound}
\end{equation}
where \(M=\int_\Omega u\,dx\).  If \(a(M,\tau,\Omega)<2/\chi\), then
\(\mathcal F_{\mathrm{PES}}\) is bounded from below.

\begin{lemma}[Entropy--energy duality]
\label{lem:gibbs-duality}
Let \(\Omega\) be bounded, \(u\ge0\) with \(\int_\Omega u\,dx=M>0\), and let
\(g\in L^\infty(\Omega)\). Then
\[
   \int_\Omega u\,g\,dx
   \le
   \int_\Omega u\log u\,dx
   +M\log\!\Bigl(\tfrac1M\!\int_\Omega e^{g}\,dx\Bigr).
\]
Equality holds when \(u=M e^{g}/\!\int_\Omega e^{g}\).
\end{lemma}

\begin{proof}
Apply Jensen's inequality to the probability density \(\rho=u/M\) and the
convex function \(s\mapsto s\log s\), or equivalently use the Gibbs variational
principle \(\int\rho g\le\int\rho\log\rho+\log\int e^{g}\) and substitute
\(\rho=u/M\).
\end{proof}

\begin{proposition}[Adams reduction of the free-energy lower bound]
\label{prop:adams-reduction}
Let \(N=4\) and \(u\in\mathcal A_M(\Omega)\). With \(g=\chi K_\tau u\),
Lemma~\ref{lem:gibbs-duality} gives
\[
   \frac{\chi}{2}\int_\Omega uK_\tau u\,dx
   \le\frac12\int_\Omega u\log u\,dx
   +\frac{M}{2}\log\!\Bigl(\tfrac1M\!\int_\Omega e^{\chi K_\tau u}\,dx\Bigr).
\]
Consequently \(\mathcal F_{\mathrm{PES}}\) is bounded below on
\(\mathcal A_M(\Omega)\) as soon as the exponential moment
\(\int_\Omega e^{\chi K_\tau u}\,dx\) is controlled by the entropy. Setting
\(v=K_\tau u\), this is precisely an Adams-type exponential-integrability
question for \(v\): since \(K_\tau\) is smoothing of order four, \(v\) carries
two Laplacians of regularity over \(u\), and the relevant control is the
fourth-order endpoint estimate \eqref{Adams-inequality}, with sharp exponent
\(32\pi^2\), rather than the second-order Trudinger--Moser estimate used in the
two-dimensional theory.
\end{proposition}

\begin{proposition}[Lower bound from logarithmic interaction control]
\label{prop:conditional-lower-bound}
Assume \(N=4\), \(M>0\), and suppose that every \(u\in\mathcal A_M(\Omega)\)
satisfies \eqref{abstract-log-bound} with
\(a(M,\tau,\Omega)<2/\chi\).  Then there exists
\(C=C(M,\Omega,\tau,\chi)>0\) such that
\begin{equation}
        \mathcal F_{\mathrm{PES}}[u]\ge -C
        \qquad
        \text{for all }u\in\mathcal A_M(\Omega).
\label{free-energy-lower-bound}
\end{equation}
More precisely, with \(\delta=1-\chi a(M,\tau,\Omega)/2>0\),
\begin{equation}
        \mathcal F_{\mathrm{PES}}[u]
        \ge
        \delta\int_\Omega u\log u\,dx
        -
        \frac{\chi}{2}C(M,\tau,\Omega).
\label{coercive-free-energy}
\end{equation}
\end{proposition}

\begin{proof}
Insert \eqref{abstract-log-bound} into
\eqref{PES-free-energy-section3}.  This gives
\[
        \mathcal F_{\mathrm{PES}}[u]
        \ge
        \left(1-\frac{\chi}{2}a(M,\tau,\Omega)\right)
        \int_\Omega u\log u\,dx
        -
        \frac{\chi}{2}C(M,\tau,\Omega).
\]
Since \(a(M,\tau,\Omega)<2/\chi\), the coefficient of the entropy is positive.
\end{proof}

\begin{conjecture}[Sharp four-dimensional PES threshold]
\label{conj:sharp-PES-threshold}
In the four-dimensional critical regime, the sharp logarithmic interaction
constant should yield the threshold
\begin{equation}
        M_*=\frac{64\pi^2\tau}{\chi}.
\label{conjectured-critical-mass}
\end{equation}
Equivalently, one expects \(\mathcal F_{\mathrm{PES}}\) to be bounded below on
\(\mathcal A_M(\Omega)\) for \(M<M_*\), to lose compactness at \(M=M_*\), and
to be unbounded below for \(M>M_*\).  This statement is conjectural unless the
corresponding sharp \(K_\tau\)-adapted logarithmic HLS/Adams inequality is
proved.
\end{conjecture}

\begin{remark}[The explicit Adams--log-HLS bridge to \(M_*\)]
\label{rem:adams-loghls-bridge}
The candidate threshold is fixed by the following chain, in which Adams is the
operative inequality and the two-dimensional case is recovered as the
second-order shadow, not the justification.

The direct route to the sharp constant is a logarithmic
Hardy--Littlewood--Sobolev inequality adapted to the kernel of \(K_\tau\): one
seeks the optimal \(\Lambda\) such that, for \(\int_\Omega u=M\),
\[
   \frac{1}{8\pi^2\tau}\iint_{\Omega\times\Omega}
        u(x)u(y)\log\frac{1}{|x-y|}\,dx\,dy
   \le
   \Lambda(M,\tau)\int_\Omega u\log u\,dx+C .
\tag{$\star$}
\]
By the duality of logarithmic HLS and exponential-class inequalities
\cite{carlen1992competing}, the sharp constant in \((\star)\) is the dual of
the sharp Adams constant \(32\pi^2\) in \eqref{Adams-inequality}, exactly as the
sharp two-dimensional log-HLS constant is dual to the Trudinger--Moser constant
\(4\pi\). Tracking that duality through the kernel coefficient
\(\frac{1}{8\pi^2\tau}\) identified in
Lemma~\ref{lem:log-kernel-four-dim} yields the sign-change mass
\(M_*=64\pi^2\tau/\chi\) of Lemma~\ref{lem:candidate-mass}. Thus the role of
Adams is not analogical: its sharp constant is, by duality, the constant in
\((\star)\), and \((\star)\) is the inequality whose optimal form
Conjecture~\ref{conj:sharp-PES-threshold} asserts. The open step is the
sharp-constant transfer through this duality for the operator
\(K_\tau\) on \(\Omega\) with the prescribed boundary conditions.
\end{remark}

A lower bound on \(\mathcal F_{\mathrm{PES}}\) is useful only if it yields
compactness.  In the critical dimension \(N=4\), compactness may fail through
concentration of mass at points.  This is the same structural phenomenon as in
two-dimensional KS, with the logarithmic kernel now arising from the
fourth-order operator.

\begin{proposition}[Tightness from bounded entropy and interaction]
\label{prop:tightness}
Let \(N=4\), let \(M>0\), and let \(\{u_k\}\subset\mathcal A_M(\Omega)\) satisfy
\(\sup_k\int_\Omega u_k\log u_k\,dx<\infty\).  If \(\Omega\) is bounded, then
\(\{u_k\}\) is uniformly integrable and, up to a subsequence,
\(u_k\rightharpoonup u\) weakly in \(L^1(\Omega)\).  If \(\Omega=\mathbb R^4\),
the same conclusion holds locally, and global tightness follows once the
sequence also satisfies a uniform moment bound, for example
\(\sup_k\int_{\mathbb R^4}|x|^2u_k(x)\,dx<\infty\).
\end{proposition}

\begin{proof}
On bounded domains, the de la Vallée-Poussin criterion applies because
\(s\mapsto s\log s\) is superlinear.  Hence the sequence is uniformly
integrable, and Dunford--Pettis gives weak compactness in \(L^1\).  On
\(\mathbb R^4\), the argument is local; a moment bound prevents loss of mass at
spatial infinity.
\end{proof}

The preceding proposition controls diffuse loss of compactness but does not
exclude concentration.  The logarithmic kernel identifies the possible
concentration mode.

\begin{proposition}[Critical concentration alternative]
\label{prop:concentration-alternative}
Let \(N=4\), let \(\{u_k\}\subset\mathcal A_M(\Omega)\), and assume that
\(\sup_k\mathcal F_{\mathrm{PES}}[u_k]<\infty\) while
\(\sup_k\int_\Omega u_k\log u_k\,dx=\infty\).  Then any loss of compactness is
caused by concentration of mass.  More precisely, after passing to a
subsequence, the measures \(u_k\,dx\) converge weakly-* to
\begin{equation}
        u\,dx+\sum_{\ell\in\mathcal I}m_\ell\delta_{x_\ell},
\label{measure-decomposition}
\end{equation}
where \(\mathcal I\) is at most countable, \(m_\ell>0\), and
\(\int_\Omega u\,dx+\sum_{\ell\in\mathcal I}m_\ell=M\).  If a sharp threshold
\(M_*\) is available, then any concentration mass must satisfy
\(m_\ell\ge M_*\) in the usual concentration-compactness sense.
\end{proposition}

\begin{proof}
The measure decomposition is the standard concentration-compactness
alternative for bounded nonnegative measures of fixed mass.  The logarithmic
singularity \eqref{kernel-log-four} implies that the only mechanism capable of
forcing divergence of the entropy while keeping mass fixed is collapse at
points.  The final quantization statement requires the sharp local lower bound
at mass below \(M_*\); hence it is conditional on the availability of the sharp
inequality.
\end{proof}

\begin{remark}[No unsupported sharpness claim]
\label{rem:no-unsupported-sharpness}
Proposition~\ref{prop:conditional-lower-bound} is a proved lower-bound
criterion under an explicitly stated logarithmic interaction estimate.
Conjecture~\ref{conj:sharp-PES-threshold} is the expected sharp version
suggested by the exact concentration computation in
Lemma~\ref{lem:candidate-mass}.  The two statements should not be conflated.
\end{remark}

The preceding results show that the PES critical regime is governed by a
four-dimensional entropy--logarithmic-interaction balance.  The free energy is
monotone along smooth solutions by Proposition~\ref{prop:PES-energy-dissipation}.
A subcritical logarithmic HLS/Adams estimate gives a lower bound by
Proposition~\ref{prop:conditional-lower-bound}.  The leading concentration
calculation gives the candidate sharp mass \(M_*=64\pi^2\tau/\chi\), but this
constant becomes a theorem only if the sharp \(K_\tau\)-adapted inequality is
proved.  Finally, compactness can fail only through point concentration, as
described by Proposition~\ref{prop:concentration-alternative}.

Thus the PES model should be analyzed in the four-dimensional Adams or logarithmic
class, not in a five-dimensional mass-critical class.  The technical core of
the sharp theory is the following problem: prove the optimal version of
\eqref{abstract-log-bound}, identify whether its sharp constant gives
\(M_*=64\pi^2\tau/\chi\), and determine whether loss of compactness at the
threshold is realized by point concentration.

\section{The Mixed Elliptic--Parabolic Cascade and the Volterra Memory Mechanism}
\label{sec:MEP-volterra}

We next analyze the mixed elliptic--parabolic cascade.  The purpose of this
section is to make clear that MEP is not a static fourth-order elliptic
chemotaxis model.  Its signal field is generated by a Volterra memory
operator.  For every positive time lag this operator has parabolic smoothing,
but its near-diagonal part retains the classical KS drift
singularity.  Thus MEP requires mixed space--time estimates and cannot be
placed in the PES critical class by formal analogy.

Throughout this section, \(A:=I-\Delta\), equipped with the appropriate
realization on \(\Omega\): the Neumann realization on a bounded smooth domain,
the periodic realization on \(\mathbb T^N\), or the usual Bessel operator on
\(\mathbb R^N\).  The MEP system is
\[
\begin{cases}
    \partial_tu=\Delta u-\chi\nabla\cdot(u\nabla c),\\
    0=\Delta w-w+u,\\
    \partial_tc=\Delta c-c+w.
\end{cases}
\tag{MEP}
\label{MEP-section4}
\]
Equivalently, \(Aw=u\) and \(\partial_tc+Ac=w\).

The elliptic equation \(0=\Delta w-w+u\) gives \(w=A^{-1}u\). 
Because the signal equation is parabolic rather than elliptic, $c$ cannot be recovered from $u$ by an instantaneous operator; solving $\partial_t c + Ac = A^{-1}u$ by Duhamel produces instead the time-nonlocal map:
\begin{equation}
        c(t)=e^{-tA}c_0+\int_0^t e^{-(t-s)A}A^{-1}u(s)\,ds.
\label{MEP-Duhamel}
\end{equation}
Thus
\begin{equation}
        \nabla c(t)
        =
        \nabla e^{-tA}c_0
        +
        \int_0^t\nabla e^{-(t-s)A}A^{-1}u(s)\,ds.
\label{MEP-gradient-Duhamel}
\end{equation}

\begin{proposition}[Volterra representation]
\label{prop:MEP-volterra-representation}
Let \(u\) and \(c\) be sufficiently regular solutions of \eqref{MEP-section4}.
Then \(c\) and \(\nabla c\) are given by \eqref{MEP-Duhamel} and
\eqref{MEP-gradient-Duhamel}.  In particular, the map \(u\mapsto\nabla c\) is
a time-dependent Volterra operator, not a static elliptic operator.
\end{proposition}

\begin{proof}
Since \(Aw=u\), the parabolic signal equation becomes
\(\partial_tc+Ac=A^{-1}u\).  Solving this inhomogeneous linear equation by the
semigroup \(e^{-tA}\) gives \eqref{MEP-Duhamel}.  Taking a spatial gradient
gives \eqref{MEP-gradient-Duhamel}.
\end{proof}

On \(\mathbb R^N\), the memory part of the drift has Fourier multiplier
\begin{equation}
        m(\xi,\theta)
        =
        \frac{i\xi e^{-\theta(1+|\xi|^2)}}{1+|\xi|^2},
        \qquad \theta=t-s>0.
\label{MEP-multiplier-section4}
\end{equation}
For fixed \(\theta>0\), the factor \(e^{-\theta(1+|\xi|^2)}\) gives strong
high-frequency smoothing.  However, the singular behavior relevant to the
Volterra integral occurs as \(s\uparrow t\), equivalently
\(\theta\downarrow0\).  In that limit,
\begin{equation}
        m(\xi,\theta)\longrightarrow \frac{i\xi}{1+|\xi|^2}.
\label{MEP-near-diagonal}
\end{equation}
The multiplier \(i\xi(1+|\xi|^2)^{-1}\) has high-frequency order \(-1\), which
is exactly the order of the classical KS drift
\(\nabla(I-\Delta)^{-1}u\). Thus the near-diagonal singularity of MEP is
classical rather than fourth order. We formalize this idea by the following proposition.

\begin{proposition}[Frozen-lag drift is uniformly classical-KS bounded]
\label{prop:frozen-lag-bound}
Let \(1<q<\infty\) and \(\theta>0\). The frozen-lag drift operator
\(T_\theta:=\nabla e^{-\theta A}A^{-1}\) satisfies
\[
   \|T_\theta f\|_{L^q(\Omega)}\le C(q,\Omega)\,\|f\|_{L^q(\Omega)},
   \qquad\text{uniformly in }\theta>0,
\]
and \(T_\theta\to\nabla A^{-1}\) strongly on \(L^q(\Omega)\) as
\(\theta\downarrow0\). The limit \(\nabla A^{-1}=\nabla(I-\Delta)^{-1}\) is
exactly the classical Keller--Segel drift operator, bounded on \(L^q\) by the
Mikhlin multiplier theorem. Hence the near-diagonal (\(\theta\downarrow0\))
behaviour of the MEP drift is that of the order-\((-1)\) KS drift, made precise
as an \(L^q\)-bound uniform in the lag.
\end{proposition}

\begin{proof}
On \(\mathbb R^N\), the Fourier multiplier of \(T_\theta\) is
\[
        m_\theta(\xi)
        =
        \frac{i\xi e^{-\theta(1+|\xi|^2)}}{1+|\xi|^2}.
\]
For \(0<\theta\le1\), the family \(\{m_\theta\}_{\theta>0}\) satisfies the
Mikhlin bounds
\[
        |\xi|^{|\alpha|}
        |\partial_\xi^\alpha m_\theta(\xi)|
        \le C_\alpha,
        \qquad |\alpha|\le [N/2]+1,
\]
with constants independent of \(\theta\).  Indeed the exponential factor only
improves high frequencies, while the limiting multiplier
\(i\xi(1+|\xi|^2)^{-1}\) is a standard Bessel--Riesz multiplier of order
\(-1\).  Hence \(T_\theta\) is bounded on \(L^q(\mathbb R^N)\), uniformly in
\(\theta>0\), for every \(1<q<\infty\).  Moreover,
\(m_\theta(\xi)\to i\xi(1+|\xi|^2)^{-1}\) pointwise as \(\theta\downarrow0\),
and the uniform multiplier bounds imply strong convergence on \(L^q\) by
density from the Schwartz class.

On \(\mathbb T^N\) the same conclusion follows from the corresponding
periodic multiplier theorem.  On a bounded smooth Neumann domain it follows
from the \(H^\infty\)-functional calculus and standard \(L^q\)-boundedness of
Bessel potential operators associated with the Neumann Laplacian.
\end{proof}

\begin{remark}[Why MEP is not PES]
\label{rem:MEP-not-PES}
For PES, the eliminated signal is \(c=K_\tau u\), with \(K_\tau\) of order
\(-4\), so \(\nabla K_\tau\) has order \(-3\).  For MEP, the memory kernel
has multiplier \eqref{MEP-multiplier-section4}, whose near-diagonal limit has
order \(-1\).  Therefore MEP cannot be classified as a fourth-order elliptic
system by replacing the Volterra operator with a static resolvent.
\end{remark}

The preceding multiplier computation rules out a purely static fourth-order
scaling classification for MEP.  In particular, the MEP equations do not by
themselves justify a three-dimensional mass-critical claim.  
The question is instead whether the Volterra memory improves time-integrated
estimates enough to alter the classical KS critical scenario.  This
is a space--time regularization problem, not a static kernel problem.

Accordingly, any critical theory for MEP should be formulated in mixed norms,
for example \(u\in L^p(0,T;L^q(\Omega))\), together with estimates for
\begin{equation}
        \mathcal V u(t)
        :=
        \int_0^t\nabla e^{-(t-s)A}A^{-1}u(s)\,ds.
\label{Volterra-drift-operator}
\end{equation}
The admissible exponents must reflect both the heat smoothing for positive time
lags and the singularity as \(s\uparrow t\).

We use the following heat-Bessel estimate.  Let \(1<q\le s<\infty\), and define
\begin{equation}
        \beta(q,s):=
        \left(
        \frac{N}{2}\left(\frac1q-\frac1s\right)-\frac12
        \right)_+.
\label{beta-qs}
\end{equation}
Then, for \(0<\theta\le T\),
\begin{equation}
        \|\nabla e^{-\theta A}A^{-1}f\|_{L^s}
        \le
        C\,\theta^{-\beta(q,s)}\|f\|_{L^q},
\label{heat-bessel-estimate}
\end{equation}
where \(C=C(N,q,s,\Omega,\tau,T)\), and the estimate is understood in the
standard semigroup sense on bounded domains, tori, or \(\mathbb R^N\).  The
case \(\beta(q,s)=0\) includes the boundedness of the Bessel-Riesz operator
\(\nabla A^{-1}\) on \(L^q\), while \(\beta(q,s)>0\) measures the additional
heat smoothing needed to reach \(L^s\).

\begin{proposition}[Space--time smoothing estimate]
\label{prop:MEP-spacetime}
Let \(T>0\), \(1<q\le s<\infty\), and let \(\beta=\beta(q,s)<1\).  Suppose
\(1\le p,r,a\le\infty\) satisfy \(1+1/r=1/a+1/p\) and \(a\beta<1\), with the
usual convention that \(a=\infty\) is allowed only when \(\beta=0\).  Then the
Volterra drift operator \(\mathcal V\) in \eqref{Volterra-drift-operator}
satisfies
\begin{equation}
        \|\mathcal V u\|_{L^r(0,T;L^s)}
        \le
        C\,T^{\frac1a-\beta}\|u\|_{L^p(0,T;L^q)}.
\label{MEP-spacetime-estimate}
\end{equation}
Here \(C=C(N,p,q,r,s,\Omega,T)\), with the usual uniform local-in-time
interpretation on \(\mathbb R^N\).
\end{proposition}

\begin{proof}
By \eqref{heat-bessel-estimate},
\[
        \|\mathcal V u(t)\|_{L^s}
        \le
        C\int_0^t(t-s)^{-\beta}\|u(s)\|_{L^q}\,ds.
\]
The kernel \(k(t)=t^{-\beta}\mathbf 1_{(0,T)}(t)\) belongs to \(L^a(0,T)\)
precisely when \(a\beta<1\), and
\(\|k\|_{L^a(0,T)}\le C T^{1/a-\beta}\).  Young's convolution inequality on
\((0,T)\), with \(1+1/r=1/a+1/p\), gives \eqref{MEP-spacetime-estimate}.
\end{proof}

\begin{theorem}[Volterra drift estimate]
\label{thm:MEP-volterra-drift}
Let \(T>0\), \(1<q\le s<\infty\), and let \(p,r,a\) satisfy the hypotheses of
Proposition~\ref{prop:MEP-spacetime}.  If \(c_0\in W^{1,s}(\Omega)\) and
\(u\in L^p(0,T;L^q(\Omega))\), then the MEP signal satisfies
\begin{equation}
        \|\nabla c\|_{L^r(0,T;L^s)}
        \le
        C\left(
        T^{1/r}\|\nabla c_0\|_{L^s}
        +
        T^{\frac1a-\beta(q,s)}
        \|u\|_{L^p(0,T;L^q)}
        \right).
\label{MEP-main-volterra-estimate}
\end{equation}
In particular, the memory part gains parabolic smoothing only through the
time-convolution estimate, and this smoothing is limited by the
near-diagonal singularity encoded in \(\beta(q,s)\).
\end{theorem}

\begin{proof}
Use \eqref{MEP-gradient-Duhamel}.  The initial term satisfies
\(\|\nabla e^{-tA}c_0\|_{L^s}\le C\|\nabla c_0\|_{L^s}\), hence contributes
\(C T^{1/r}\|\nabla c_0\|_{L^s}\) in \(L^r(0,T)\).  The memory term is bounded
by Proposition~\ref{prop:MEP-spacetime}.
\end{proof}

The natural continuation criterion for MEP is not a static mass threshold, but
control of a space--time norm strong enough to bound the chemotactic drift.  A
robust version is the following.

\begin{proposition}[Continuation criterion]
\label{prop:MEP-continuation}
Let \((u,c)\) be a classical MEP solution on \([0,T_{\max})\), with
\(u_0\in L^1\cap L^\infty\) and \(c_0\in W^{1,\infty}\).  If, for every
\(T<T_{\max}\),
\begin{equation}
        \|u\|_{L^\infty(0,T;L^\infty)}
        +
        \|\nabla c\|_{L^1(0,T;W^{1,\infty})}
        <\infty,
\label{MEP-continuation-norm}
\end{equation}
then the solution extends beyond \(T_{\max}\).  Consequently, if
\(T_{\max}<\infty\), then at least one of the quantities in
\eqref{MEP-continuation-norm} diverges as \(T\uparrow T_{\max}\).
\end{proposition}

\begin{proof}
Under \eqref{MEP-continuation-norm}, the equation for \(u\) is a uniformly
parabolic drift-diffusion equation with bounded coefficients and controlled
spatial derivatives.  Standard parabolic continuation gives a uniform
\(L^\infty\)-based classical norm up to \(T_{\max}\), and the linear parabolic
equation for \(c\) then gives the corresponding continuation of the signal.
This contradicts maximality unless \(T_{\max}=\infty\).
\end{proof}

\begin{remark}[Critical closure remains open]
\label{rem:MEP-critical-open}
Proposition~\ref{prop:MEP-continuation} is deliberately stronger than a
critical criterion.  To obtain a genuinely critical MEP theory, one would need
to replace \eqref{MEP-continuation-norm} by a scale-invariant space--time norm
closed under the nonlinear map \(u\mapsto\nabla\cdot(u\nabla c)\), using
Theorem~\ref{thm:MEP-volterra-drift}.  Such a result would be a memory-based
critical theory, not a fourth-order elliptic threshold theorem.
\end{remark}

\subsection{Small-data theory and the absence of a proved three-dimensional threshold}
\label{ssec:small_data}
The estimates above suggest a possible fixed-point framework for mild
solutions.  Define
\begin{equation}
        u(t)=e^{t\Delta}u_0
        -
        \chi\int_0^t\nabla\cdot e^{(t-s)\Delta}
        \bigl(u(s)\nabla c(s)\bigr)\,ds,
\label{MEP-mild-u}
\end{equation}
with \(\nabla c\) given by \eqref{MEP-gradient-Duhamel}.  A small-data theorem
may be proved in a Banach space \(X_T\) once the product estimate
\(u\nabla c\in L^{\tilde p}(0,T;L^{\tilde q})\) closes the heat estimate in
\eqref{MEP-mild-u}.  This gives the following schematic, but mathematically
honest, formulation.

\begin{theorem}[Conditional small-data MEP theory]
\label{thm:MEP-conditional-small-data}
Let \(X_c\) be a Banach space of initial data and \(X\) a space--time solution
space such that the heat map \(u_0\mapsto e^{t\Delta}u_0\), the Volterra map
\(u\mapsto\nabla c\) in \eqref{MEP-gradient-Duhamel}, and the bilinear map
\((u,\nabla c)\mapsto \int_0^t\nabla\cdot e^{(t-s)\Delta}(u\nabla c)(s)\,ds\)
satisfy the estimates
\[
        \|e^{t\Delta}u_0\|_X\le C\|u_0\|_{X_c},
        \qquad
        \|\nabla c\|_Y\le C(\|c_0\|_{Y_0}+\|u\|_X),
\]
and
\[
        \left\|
        \int_0^t\nabla\cdot e^{(t-s)\Delta}(u\nabla c)(s)\,ds
        \right\|_X
        \le
        C\|u\|_X\|\nabla c\|_Y .
\]
Then, for \(\|u_0\|_{X_c}+\|c_0\|_{Y_0}\) sufficiently small, MEP admits a
unique mild solution in \(X\) on the corresponding time interval; if the
estimates are global in time, the solution is global.
\end{theorem}

\begin{remark}
Theorem~\ref{thm:MEP-conditional-small-data} is a fixed-point template rather
than a mass-critical threshold theorem.  It should not be interpreted as proof
that \(N=3\) is critical for MEP.  A three-dimensional mass-critical result
would require a sharp scale-invariant closure and a threshold mechanism
specific to the Volterra memory operator.
\end{remark}

\section{Global Theory, Blow-Up Scenarios, Open Problems, and Final Revision Strategy}
\label{sec:global-open-final}

This final section summarizes the rigorous consequences of the preceding structural analysis and fixes the final theorem--conjecture architecture of the paper.  The main publishable conclusion is not a universal threshold theorem for all indirect chemotaxis systems, but the classification:
\begin{equation*}
\begin{aligned}
&\mathrm{PES}\Rightarrow\text{four-dimensional Adams/logarithmic criticality},\\[2pt]
&\mathrm{MEP}\Rightarrow\ \text{Volterra-memory criticality with classical near-diagonal drift}.
\end{aligned}
\end{equation*}
Thus PES and MEP belong to different analytical classes.  PES is governed by a static fourth-order self-adjoint elliptic interaction, whereas MEP is governed by a time-memory operator whose near-diagonal drift remains KS type.

The results in this section should be read as a criticality classification rather than a
complete nonlinear blow-up theory.  For KS equations, scaling,
energy structure, critical mass, finite-time blow-up, infinite-time
concentration, and asymptotic profile selection are logically distinct
questions.  The present paper resolves the first two questions for the PES
cascade and identifies the correct variational threshold candidate in the
four-dimensional logarithmic regime.  It does not claim a full dynamical
classification above threshold.  In particular, finite-time blow-up,
infinite-time concentration, profile selection, and mass quantization for PES
remain open problems.

\subsection{PES global theory with threshold}

Let \(N=4\), \(M=\int_\Omega u_0\,dx\), and let
\[
        \mathcal F_{\mathrm{PES}}[u]
        =
        \int_\Omega u\log u\,dx
        -
        \frac{\chi}{2}\int_\Omega uK_\tau u\,dx,
        \qquad
        K_\tau=(I-\tau\Delta)^{-1}(I-\Delta)^{-1}.
\]
The conservative global theory should be formulated through a proved lower bound for \(\mathcal F_{\mathrm{PES}}\), not through an unsupported sharp threshold.  The minimal assumption needed is the following logarithmic interaction estimate: for \(u\in\mathcal A_M(\Omega)\),
\begin{equation}
        \int_\Omega uK_\tau u\,dx
        \le a(M,\tau,\Omega)\int_\Omega u\log u\,dx+C(M,\tau,\Omega),
        \qquad
        a(M,\tau,\Omega)<\frac{2}{\chi}.
\label{eq:subcritical-log-control-final}
\end{equation}
Under \eqref{eq:subcritical-log-control-final}, Proposition~\ref{prop:conditional-lower-bound} gives
\[
        \mathcal F_{\mathrm{PES}}[u]\ge -C(M,\Omega,\tau,\chi).
\]

\begin{theorem}[Conservative subcritical global PES theory]
\label{thm:conservative-global-PES}
Let \(N=4\), let \(\Omega\) be either \(\mathbb R^4\), \(\mathbb T^4\), or a bounded \(C^2\) domain with the boundary conditions stated in Section~\ref{sec:intro-main-results}, and let \(u_0\ge0\) satisfy \(u_0\in L^1(\Omega)\cap L^\infty(\Omega)\), \(\int_\Omega u_0\,dx=M\).  Assume the subcritical logarithmic control \eqref{eq:subcritical-log-control-final}.  Then the PES free energy is bounded from below along smooth solutions, and
\[
        \sup_{0<t<T}\int_\Omega u(t)\log u(t)\,dx
        +
        \int_0^T\!\!\int_\Omega
        u\left|\nabla(\log u-\chi K_\tau u)\right|^2dxdt
        \le C(T,u_0,\Omega,\tau,\chi)
\]
for every \(T>0\).  Consequently, any approximation scheme preserving mass, positivity, and the dissipation identity admits global entropy-level compactness.  If the corresponding local classical theory has a continuation criterion controlled by these entropy bounds and parabolic smoothing, then the solution is global.
\end{theorem}

\begin{remark}
Theorem~\ref{thm:conservative-global-PES} is intentionally conservative.  It proves global entropy compactness under a stated lower-bound hypothesis.  It does not assert a sharp critical mass unless the sharp Adams/logarithmic HLS estimate is separately proved.
\end{remark}

The stronger target result is Conjecture~\ref{conj:sharp-PES-threshold}.

\begin{remark}[Status of the critical mass constant]
\label{rem:status-critical-mass}
The constant
\[
        M_*^{\rm cand}:=\frac{64\pi^2\tau}{\chi}
\]
is not proved in this paper to be a sharp critical mass.  It is obtained from
the leading logarithmic singularity of the four-dimensional PES kernel and from
the concentration scaling of the free energy.  More precisely, if
\(u_\varepsilon(x)=\varepsilon^{-4}U((x-x_0)/\varepsilon)\) with
\(\int_{\mathbb R^4}U\,dx=M\), then \eqref{energy-scaling} holds:
\[
        \mathcal F_{\mathrm{PES}}[u_\varepsilon]
        =
        \left(
        4M-\frac{\chi M^2}{16\pi^2\tau}
        \right)\log\frac1\varepsilon+O(1).
\]
Thus the energy changes sign under concentration at
\(M=64\pi^2\tau/\chi\).  This computation identifies the natural candidate
threshold, but a sharp threshold theorem would require a sharp logarithmic
HLS or Adams-type inequality adapted to
\[
        K_\tau=(I-\tau\Delta)^{-1}(I-\Delta)^{-1}.
\]
Since such an inequality is not proved here, all sharp-threshold statements are
formulated as conjectures or conditional consequences.
\end{remark}

Above the candidate critical mass, the paper should not assert dynamical blow-up unless a blow-up construction is supplied.  What can be proved at the structural level is energy unboundedness along concentrating sequences.

\begin{proposition}[Variational instability above the candidate mass]
\label{prop:variational-instability-PES}
Let \(N=4\), and let \(M>64\pi^2\tau/\chi\).  Then there exists a sequence
\(\{u_\varepsilon\}\subset\mathcal A_M(\Omega)\) such that
\(u_\varepsilon\,dx\rightharpoonup M\delta_{x_0}\) weakly-* as measures and
\[
        \mathcal F_{\mathrm{PES}}[u_\varepsilon]\to-\infty
        \qquad\text{as }\varepsilon\downarrow0 .
\]
Consequently, the supercritical PES free energy is variationally unstable.
This statement does not imply finite-time blow-up of the PES flow.
\end{proposition}

\begin{proof}
Choose \(u_\varepsilon(x)=\varepsilon^{-4}U((x-x_0)/\varepsilon)\), where \(U\ge0\) is smooth, compactly supported, and \(\int_{\mathbb R^4}U\,dx=M\).  By Lemma~\ref{lem:candidate-mass},
\[
        \mathcal F_{\mathrm{PES}}[u_\varepsilon]
        =
        \left(4M-\frac{\chi M^2}{16\pi^2\tau}\right)\log\frac1\varepsilon+O(1).
\]
The coefficient is negative exactly when \(M>64\pi^2\tau/\chi\).  Hence the energy tends to \(-\infty\).
\end{proof}

\begin{problem}[Nonlinear dynamics of supercritical PES]
\label{prob:supercritical-PES-dynamics}
Let \(N=4\) and \(M>64\pi^2\tau/\chi\).  Determine whether smooth PES
solutions with supercritical mass can develop finite-time blow-up.  If
finite-time blow-up occurs, identify the blow-up rate, concentration profile,
and mass quantization.  If finite-time blow-up is ruled out, determine whether
supercritical solutions may instead exhibit infinite-time concentration or
grow-up.  In either case, derive the analogue of the virial, modulation, or
concentration-compactness mechanisms known for the classical two-dimensional
KS equation.
\end{problem}

\begin{remark}[Energy descent is not blow-up]
\label{rem:energy-not-blowup}
Proposition~\ref{prop:variational-instability-PES} proves variational instability above the candidate threshold.  It does not by itself prove finite-time blow-up for the PES flow.  A dynamical blow-up theorem would require an additional mechanism, such as a virial identity, a radial comparison argument, or construction of a self-similar concentrating solution.
\end{remark}

A natural dynamical target is therefore:

\begin{problem}[PES blow-up or concentration dynamics]
\label{prob:PES-dynamics}
For \(N=4\) and \(M>M_*\), determine whether PES solutions exhibit finite-time blow-up, infinite-time concentration, or merely variational energy descent.  In radial geometries, derive or disprove a virial-type criterion capable of turning energy negativity into dynamical concentration.
\end{problem}

\subsection{MEP continuation and possible critical scenarios}

For MEP, the global theory must retain the Volterra memory structure.  The eliminated signal satisfies
\[
        \nabla c(t)=\nabla e^{-tA}c_0+\int_0^t\nabla e^{-(t-s)A}A^{-1}u(s)\,ds.
\]
The near-diagonal multiplier has order \(-1\), so MEP cannot be assigned a mass-critical dimension by static fourth-order scaling.  The appropriate continuation theory should be expressed in mixed space--time norms.

\begin{definition}[Admissible MEP continuation pair]
\label{def:MEP-admissible-pair}
A pair \((p,q)\) is called an admissible MEP continuation pair on \((0,T)\) if the Volterra estimate of Theorem~\ref{thm:MEP-volterra-drift}, combined with parabolic estimates for the \(u\)-equation, implies
\[
        \|u\|_{L^p(0,T;L^q(\Omega))}<\infty
        \quad\Longrightarrow\quad
        \|\nabla c\|_{L^1(0,T;W^{1,\infty}(\Omega))}<\infty
\]
and hence closes the classical continuation criterion.
\end{definition}

\begin{proposition}[Conditional MEP continuation criterion]
\label{prop:MEP-critical-continuation-final}
Let \((u,c)\) be a classical MEP solution on \([0,T_{\max})\), with \(u_0\in L^1\cap L^\infty\) and \(c_0\in W^{1,\infty}\).  Suppose \((p,q)\) is an admissible MEP continuation pair and
\[
        \|u\|_{L^p(0,T_{\max};L^q(\Omega))}<\infty.
\]
Then \(T_{\max}=\infty\).  Equivalently, if \(T_{\max}<\infty\), every admissible continuation norm must diverge as \(t\uparrow T_{\max}\).
\end{proposition}

\begin{proof}
By admissibility, the assumed \(L^p_tL^q_x\) bound yields
\(\|\nabla c\|_{L^1(0,T_{\max};W^{1,\infty})}<\infty\).  The parabolic
drift-diffusion equation for \(u\) then satisfies the continuation criterion
of Proposition~\ref{prop:MEP-continuation}.  Hence the solution extends beyond
\(T_{\max}\), a contradiction unless \(T_{\max}=\infty\).
\end{proof}

The central unresolved question is:

\begin{problem}[Does memory change criticality?]
\label{prob:memory-criticality-final}
Does the Volterra memory in MEP merely delay classical KS
concentration, or can it genuinely change the critical threshold?  In
particular, is there a scale-invariant mixed norm for MEP whose finiteness
yields global continuation and whose failure corresponds to concentration?
A three-dimensional mass-critical theorem should not be claimed unless this
memory mechanism is proved.
\end{problem}

\subsection{Nonlinear dynamics beyond scaling}
\label{subsec:nonlinear-dynamics-beyond-scaling}

The preceding results should be interpreted as a structural and variational
classification, not as a complete nonlinear dynamics theory.  Scaling identifies
the critical dimension and the natural critical space, but it does not by itself
prove finite-time blow-up, infinite-time concentration, convergence to
asymptotic profiles, or quantization of concentrating mass.  These phenomena
require additional dynamical mechanisms.

For the classical parabolic--elliptic KS system in two dimensions,
the mass-critical scaling is only the starting point.  The sharp nonlinear
theory also uses the logarithmic HLS inequality,
free-energy dissipation, virial identities, concentration compactness, and
refined profile analysis.  In particular, the threshold \(8\pi/\chi\) separates
global subcritical behavior from supercritical collapse in the standard
normalization, while at the critical mass one may observe infinite-time
concentration rather than finite-time blow-up. 
Thus even in the classical second-order model, scaling and nonlinear dynamics are distinct levels of the
analysis.

The same distinction is essential for the PES model.  The present paper proves
that the PES interaction has principal order \(-4\), that the drift has order
\(-3\), that the critical amplitude is \(\alpha=4\), and that the
mass-critical dimension is \(N=4\).  It also proves the free-energy identity
\[
        \frac{d}{dt}\mathcal F_{\mathrm{PES}}[u(t)]
        =
        -
        \int_\Omega
        u\left|\nabla(\log u-\chi K_\tau u)\right|^2dx
        \le0
\]
for smooth positive solutions.  Moreover, concentration testing in dimension
four gives the candidate threshold
\[
        M_*=\frac{64\pi^2\tau}{\chi}.
\]
These statements establish the critical architecture.  They do not,
however, prove that supercritical PES solutions blow up in finite time.

The nonlinear dynamical alternatives for supercritical PES data remain open.
For \(M>M_*\), the free energy is unbounded from below along concentrating
sequences.  This shows variational instability, but it does not determine the
evolutionary scenario.  A solution may in principle develop finite-time blow-up,
concentrate only as \(t\to\infty\), approach a slowly concentrating critical
profile, or remain globally regular while descending along increasingly
concentrated configurations.  Distinguishing these alternatives requires tools
not supplied by scaling alone, such as a fourth-order virial identity, a radial
comparison principle, a modulation analysis near concentrating profiles, or a
sharp concentration-compactness theorem adapted to \(K_\tau\).

Accordingly, the claims of the paper are deliberately separated as follows.
The operator classification, critical scaling, logarithmic kernel structure,
and energy dissipation are proved.  The sharp value of the critical mass is
identified as a concentration-scaling candidate.  Global entropy compactness
below threshold is conditional on a sharp \(K_\tau\)-adapted logarithmic
HLS/Adams inequality.  Finite-time blow-up, infinite-time concentration,
asymptotic profile selection, and concentration quantization are left as open
problems in the nonlinear dynamics of the fourth-order PES flow.

We state the theorem--conjecture architecture.
\begin{enumerate}
    \item \textbf{Theorems and propositions:} spectral positivity of
    \(K_\tau\), fourth-order homogeneity of PES, PES entropy dissipation, PES
    scaling \(q_c=N/4\), four-dimensional logarithmic kernel, Volterra
    representation of MEP, near-diagonal multiplier asymptotics, and
    space--time Volterra estimates.

    \item \textbf{Conditional theorems:} PES global entropy compactness under a
    proved subcritical logarithmic interaction bound; MEP continuation under an
    admissible mixed-norm closure criterion.

    \item \textbf{Conjectures:} sharp PES critical mass \(M_*=64\pi^2\tau/\chi\)
    and any memory-improved MEP critical threshold.

    \item \textbf{Open problems:} dynamical PES blow-up above threshold,
    concentration quantization at \(M_*\), and whether MEP has any
    nonclassical memory-induced critical dimension.
\end{enumerate}

\section{A Numerical Diagnostic for the Operator Classification}
\label{sec:numerical-diagnostic}

This section gives a small numerical diagnostic for the distinction between
direct KS signalling, the PES cascade, and the MEP cascade. 
The purpose is to test whether the three closures display the
operator-level behavior predicted by the preceding analysis. Namely, direct
KS signalling should generate stronger aggregation, PES should act
as a stronger spatial filter, and MEP should approach the quasi-steady PES
regime as the signal relaxation parameter tends to zero.

The numerical test is motivated by two facts.  First, numerical KS
studies commonly monitor peak-density growth as an indicator of aggregation or
blow-up tendency.  Second, indirect signal production can change the aggregation
scenario substantially.  In particular, the indirect signal-production model of
Tao and Winkler \cite{tao2017critical} exhibits a critical-mass phenomenon associated with
infinite-time aggregation rather than finite-time blow-up.  Thus the present
computation is used only as a structural diagnostic, not as a replacement for a
nonlinear threshold theorem.

\subsection{Periodic pseudo-spectral setting}

We work on the two-dimensional flat torus $\Omega=[0,2\pi)^2$ with periodic boundary conditions.  The population equation is
\begin{equation}
\label{eq:numerical-u-equation}
        \partial_t u
        =
        \Delta u-\chi\nabla\cdot(u\nabla c).
\end{equation}
The same initial condition \(u_0\) is evolved under three different signal
closures.

The direct KS closure is taken in Bessel-regularized form,
\begin{equation}
\label{eq:numerical-KS}
        c=(I-\Delta)^{-1}u.
\tag{KS}
\end{equation}
The PES closure is
\begin{equation}
\label{eq:numerical-PES}
        c=(I-\tau\Delta)^{-1}(I-\Delta)^{-1}u.
\tag{PES}
\end{equation}
The MEP closure is
\begin{equation}
\label{eq:numerical-MEP}
        \varepsilon_2\partial_t c
        =
        \Delta c-c+(I-\Delta)^{-1}u.
\tag{MEP}
\end{equation}
Thus KS uses one elliptic filter, PES uses two elliptic filters in series, and
MEP uses an elliptic filter followed by finite-time relaxation of the sensed
signal.

On the torus, the closures are diagonal in Fourier variables.  If \(k\in
\mathbb Z^2\), then
\[
        \widehat c_{\rm KS}(k)
        =
        \frac{\widehat u(k)}{1+|k|^2},
\]
while
\[
        \widehat c_{\rm PES}(k)
        =
        \frac{\widehat u(k)}
        {(1+|k|^2)(1+\tau |k|^2)}.
\]
For MEP, the Fourier modes satisfy
\[
        \varepsilon_2\frac{d}{dt}\widehat c(k,t)
        =
        -(1+|k|^2)\widehat c(k,t)
        +
        \frac{\widehat u(k,t)}{1+|k|^2}.
\]
Hence, as \(\varepsilon_2\to0\), the MEP closure formally approaches the PES
closure with \(\tau=1\):
\[
        \widehat c(k,t)
        \to
        \frac{\widehat u(k,t)}{(1+|k|^2)^2}.
\]

The population equation \eqref{eq:numerical-u-equation} is discretized by a
pseudo-spectral method.  The elliptic closures are evaluated by Fourier
multipliers.  The MEP signal equation is advanced semi-implicitly in Fourier
space.  The \(u\)-equation is advanced by an explicit second-order Runge--Kutta
step with mild spectral filtering and a mass correction.  These stabilizations
are used only to make the diagnostic comparison robust; the computation is not
intended as a positivity-preserving or convergence-certified numerical method.

\subsection{Diagnostics and results}

Three diagnostics are recorded.  The first is the peak population density
\[
        U_{\max}(t):=\max_{x\in\Omega}u(x,t).
\]
The second is an aggregation time
\[
        T_{\rm agg}
        :=
        \inf\left\{
        t>0:
        U_{\max}(t)\ge 2U_{\max}(0)
        \right\}.
\]
If this threshold is not reached during the simulated time interval, we write
\(T_{\rm agg}=-\).

The third diagnostic is an effective smoothing length of the signal field,
defined by
\begin{equation}
\label{eq:effective-smoothing-length}
        \ell_{\rm eff}(c)
        :=
        \left(
        \frac{\int_\Omega |c-\bar c|^2\,dx}
             {\int_\Omega |\nabla c|^2\,dx}
        \right)^{1/2},
        \qquad
        \bar c:=|\Omega|^{-1}\int_\Omega c\,dx.
\end{equation}
This quantity is larger when the signal is dominated by longer spatial
wavelengths.  Thus \(\ell_{\rm eff}\) gives a simple numerical measure of the
spatial filtering imposed by the signal law.

The same initial mass distribution is evolved under KS, PES, and MEP closures.
For MEP, we vary
\[
        \varepsilon_2\in\{10^{-3},10^{-2},10^{-1},1\}.
\]
The final time, spatial grid, and all numerical parameters are kept fixed
across the comparisons.  The measured total mass is conserved up to round-off
error in all reported runs.

\begin{table}[ht]
\centering
\small
\renewcommand{\arraystretch}{1.15}
\caption{Diagnostic comparison of KS, PES, and MEP closures.  Here
\(T_{\rm agg}\) is the first time at which
\(\max_x u(x,t)\ge 2\max_xu(x,0)\).  The value \(T_{\rm agg}=-\)
means that the aggregation threshold was not reached during the simulated
interval.}
\label{tab:diagnostic-KS-PES-MEP}
\begin{tabular}{@{}lccccc@{}}
\toprule
Closure & \(\varepsilon_2\) & \(T_{\rm agg}\) &
\(\max_x u(x,T)\) & Peak growth &
\(\ell_{\rm eff}(c(T))\) \\
\midrule
KS  & --          & \(0.055\) & \(337.304\) & \(84.12\) & \(0.5445\) \\
PES & --          & --        & \(5.472\)   & \(1.365\)  & \(0.9579\) \\
MEP & \(10^{-3}\) & --        & \(5.459\)   & \(1.361\)  & \(0.9580\) \\
MEP & \(10^{-2}\) & --        & \(5.346\)   & \(1.333\)  & \(0.9584\) \\
MEP & \(10^{-1}\) & --        & \(4.569\)   & \(1.139\)  & \(0.9607\) \\
MEP & \(1\)       & --        & \(3.166\)   & \(0.790\)  & \(0.9551\) \\
\bottomrule
\end{tabular}
\end{table}

The direct KS closure exhibits rapid aggregation.  In this run,
\[
        T_{\rm agg}=0.055,
        \qquad
        \max_x u(x,T)=337.304,
\]
and the peak amplification factor is approximately
\[
        \frac{\max_x u(x,T)}{\max_x u(x,0)}=84.12.
\]
By contrast, the PES closure does not reach the aggregation threshold on the
same time interval.  Its final peak density is only
\[
        \max_x u(x,T)=5.472,
\]
with peak amplification factor \(1.365\).  The signal smoothing length also
separates the two closures:
\[
        \ell_{\rm eff}^{\rm KS}(c(T))\approx0.5445,
        \qquad
        \ell_{\rm eff}^{\rm PES}(c(T))\approx0.9579.
\]
Thus the PES closure produces a substantially smoother signal field, consistent
with its interpretation as a fourth-order elliptic filter.

The MEP results interpolate between the quasi-steady PES regime and a
finite-relaxation memory regime.  For \(\varepsilon_2=10^{-3}\), the MEP output
is nearly indistinguishable from PES:
\[
        \max_x u_{\rm MEP}(x,T)=5.459,
        \qquad
        \ell_{\rm eff}^{\rm MEP}(c(T))=0.9580,
\]
compared with
\[
        \max_x u_{\rm PES}(x,T)=5.472,
        \qquad
        \ell_{\rm eff}^{\rm PES}(c(T))=0.9579.
\]
This confirms the expected quasi-steady limit
\[
        \varepsilon_2\to0
        \quad\Longrightarrow\quad
        \mathrm{MEP}\to\mathrm{PES}.
\]
As \(\varepsilon_2\) increases from \(10^{-3}\) to \(1\), the final peak density
decreases from \(5.459\) to \(3.166\).  Thus, in this parameter regime, finite
signal relaxation delays or suppresses peak growth rather than enhancing
aggregation.

The computation supports the mathematical-physics distinction developed in the
paper.  KS, PES, and MEP do not behave as interchangeable chemotaxis closures.
The KS model uses a first-order drift operator
\[
        \nabla(I-\Delta)^{-1},
\]
and in the diagnostic run it produces rapid peak growth.  PES uses
\[
        \nabla(I-\tau\Delta)^{-1}(I-\Delta)^{-1},
\]
which is a stronger spatial filter.  The numerical signal field is accordingly
much smoother, and aggregation is strongly reduced on the tested time interval.
MEP retains the same quasi-steady limit as PES when the final signal relaxation
time is small, but for finite relaxation it introduces memory and history
dependence through
\[
        c(t)
        =
        e^{-t/\varepsilon_2\,A}c_0
        +
        \frac1{\varepsilon_2}
        \int_0^t e^{-(t-s)A/\varepsilon_2}A^{-1}u(s)\,ds,
        \qquad A=I-\Delta.
\]
The observed decrease in peak growth for larger \(\varepsilon_2\) is consistent
with this memory interpretation.

These numerical results test the structural prediction
that PES behaves as a static fourth-order spatial filter and that MEP approaches
PES in the fast-relaxation limit while retaining finite-time memory when
\(\varepsilon_2=\mathcal O(1)\).

\begin{figure}[ht]
\centering
\includegraphics[width=0.72\textwidth]{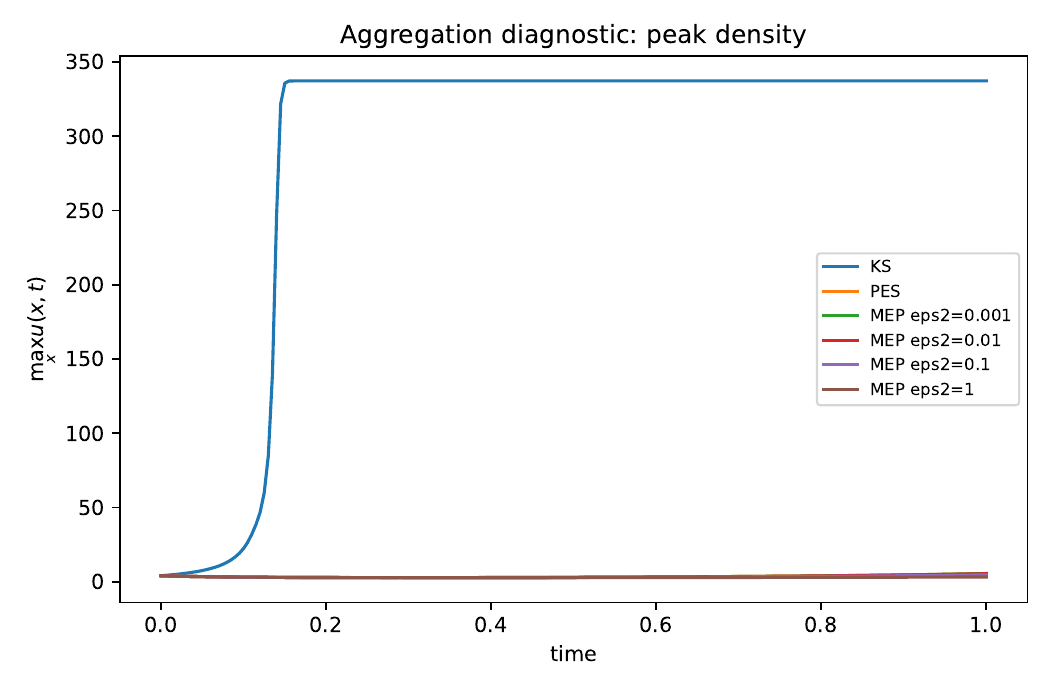}
\caption{Peak-density diagnostic for KS, PES, and MEP closures.  The same
initial population density is evolved under each signal law.  The KS closure
produces rapid peak amplification and crosses the aggregation threshold at
\(T_{\rm agg}=0.055\).  The PES closure remains strongly regularized over the
same time interval.  The MEP closure approaches PES as
\(\varepsilon_2\to0\), while larger \(\varepsilon_2\) delays or suppresses peak
growth.  The computation is a diagnostic pseudo-spectral experiment and is not
used as evidence for a sharp critical mass.}
\label{fig:peak-density-diagnostic}
\end{figure}

\begin{figure}[ht]
\centering
\includegraphics[width=0.72\textwidth]{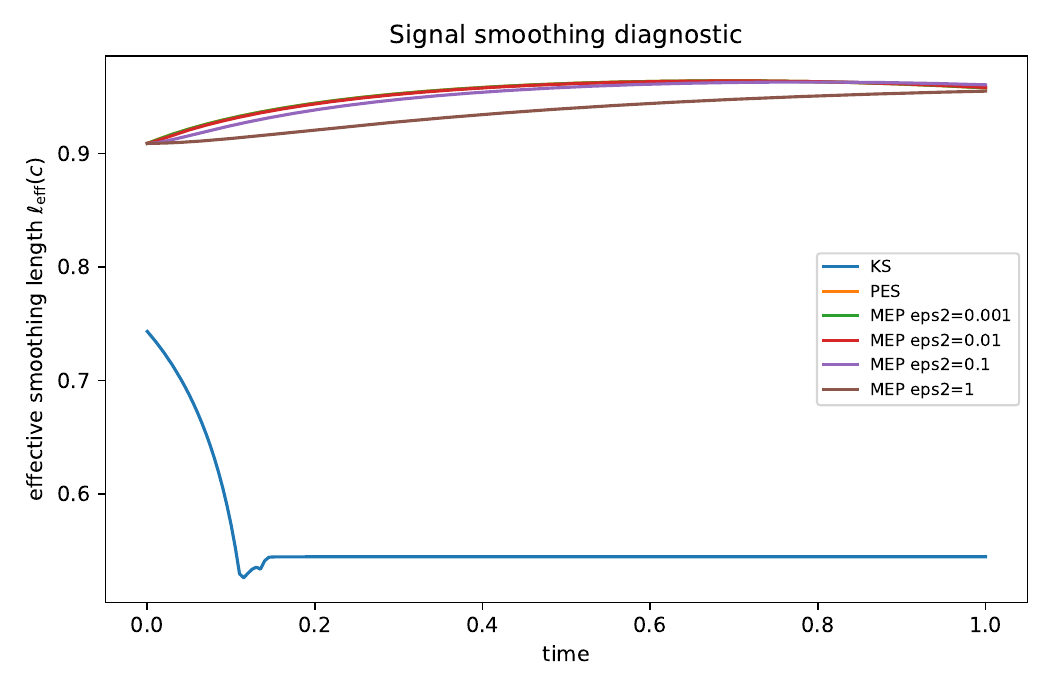}
\caption{Effective smoothing length of the sensed signal field,
\[
        \ell_{\rm eff}(c)
        =
        \left(
        \frac{\int_\Omega |c-\bar c|^2\,dx}
             {\int_\Omega |\nabla c|^2\,dx}
        \right)^{1/2}.
\]
PES and small-\(\varepsilon_2\) MEP produce substantially smoother signal fields
than the direct KS closure, consistent with the fourth-order filtering effect
of the PES operator.}
\label{fig:smoothing-length-diagnostic}
\end{figure}

\begin{figure}[ht]
\centering
\includegraphics[width=0.86\textwidth]{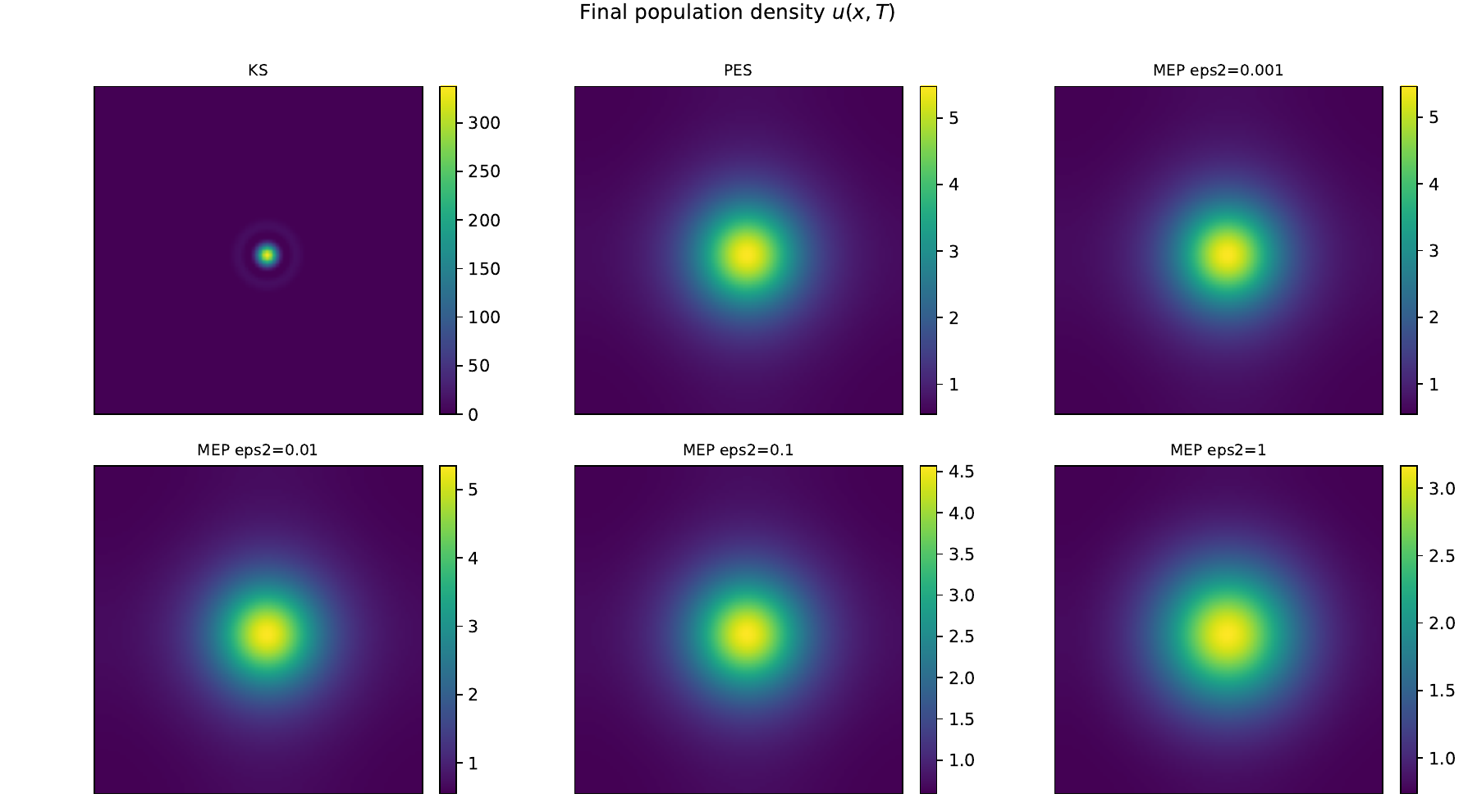}
\caption{Final population densities for the KS, PES, and MEP closures.  The
direct KS closure produces a sharply concentrated aggregate, whereas PES and
small-\(\varepsilon_2\) MEP remain spatially regularized.  Larger
\(\varepsilon_2\) introduces finite-memory effects and further suppresses peak
growth in this diagnostic run.}
\label{fig:final-fields-diagnostic}
\end{figure}

\section{Discussion}

PES admits a rigorous fourth-order structural classification: its self-adjoint elliptic
kernel has principal order \(-4\), its drift has order \(-3\), its scaling
amplitude is \(\alpha=4\), its mass-critical dimension is \(N=4\), and its
natural critical space is \(L^{N/4}\). 
In four dimensions, the PES kernel is
logarithmic, so the correct threshold mechanism is Adams/logarithmic rather
than classical KS.
MEP is fundamentally different.  Its signal is produced by a Volterra memory
operator whose near-diagonal drift has the same order as classical
KS.  Therefore MEP requires mixed space--time estimates and cannot
be assigned a new mass-critical dimension by static elliptic scaling.  The
proper open problem is whether memory changes criticality at all.

The PES--MEP distinction suggests a concrete numerical and experimental test.
One may compare three regimes: direct signalling, quasi-steady two-stage
signalling, and finite-relaxation two-stage signalling.  In nondimensional
parameters, this corresponds to varying the two signal relaxation ratios
\(\varepsilon_1/T_{\rm mig}\) and \(\varepsilon_2/T_{\rm mig}\).  When both are
small, the system should approach the PES limit and exhibit stronger spatial
smoothing of the sensed cue, consistent with a fourth-order elliptic kernel.
When the second ratio is order one, the model enters the MEP regime, and the
observable effect should be temporal delay, memory-dependent gradient formation,
and possible postponement or reshaping of aggregation.  Numerically, this can
be tested by measuring concentration rates, aggregation onset time, peak density
growth, and dependence on the total mass as \(\varepsilon_1,\varepsilon_2\) are
varied.  Experimentally, analogous tests could be designed in systems where
signal degradation, relay, or extracellular conversion rates can be perturbed.
The present work only identifies the precise quasi-steady
cascade under which such a regime is mathematically expected.

This corrected architecture is mathematically credible because it distinguishes
proved operator facts, formal scaling laws, conditional global theory,
variational concentration mechanisms, and genuinely open threshold problems.

\section*{Author’s Important Statements}
All authors are satisfied with this preprint. This research is the joint work with Jiguang Yu from Boston University, USA, Ye Liang from the University of Iowa, USA and Jilin Zhang from Imperial College London, UK.

\end{document}